\documentclass[12pt,reqno]{amsart}

\usepackage[english]{babel} 

\usepackage{amssymb}
\usepackage{graphicx}
\usepackage{stmaryrd}
\usepackage{fancyhdr}
\usepackage{environ}
\usepackage{hyperref}

\theoremstyle{plain}
\numberwithin{equation}{section}
\newtheorem{thm}{Theorem}[section]
\newtheorem{theorem}[thm]{Theorem}
\newtheorem{lemma}[thm]{Lemma}
\newtheorem{corollary}[thm]{Corollary}
\newtheorem{property}[thm]{Property}
\newtheorem{remark}[thm]{Remark}
\newtheorem{definition}[thm]{Definition}
\newtheorem{example}[thm]{Example}

\DeclareMathOperator{\lcm}{lcm}

\makeindex

\begin{document}

\title{Ring extension of entire ring with conjugation;
arithmetic in entire rings}

\author{Alexandre Laugier}

\email{laugier.alexandre@orange.fr}

\begin{abstract}
Some basic properties of the ring of integers $\mathbb{Z}$ are
extended to entire rings. In particular, arithmetic in
entire principal rings is very similar than arithmetic in the ring of integers
$\mathbb{Z}$. These arithmetic 
properties are derived from a $\star$-ring extension of the considered
entire ring (ring extension with conjugation) equipped with a real
function which is a multiplicative structure-preserving map between
two algebras. The algebra of this ring  
extension is studied in detail. Some examples of such ring extension
are given.
\end{abstract}

\maketitle

\tableofcontents

\newpage

\subsection*{{Preface}}

Arithmetic on entire rings 
which are equipped with a relation of total
order and where a relation of divisibility is defined, 
is very similar than arithmetic on the ring of integers
$\mathbb{Z}$. Indeed, the standard arithmetic properties 
in $\mathbb{Z}$ can be extended to these algebraic structures. But,
when the relation of order\index{relation of order} on elements of an
entire ring, is not
total, some arithmetic properties fails to be 
true. In particular, the greatest common divisors ($\gcd$) and the least
common divisor ($\lcm$) of two
elements in this algebraic structure, 
are not always defined. Then, the group of units of an
entire ring may be larger than in $\mathbb{Z}$. Moreover, an
equivalence relation\index{equivalence relation} 
$\sim$ defined on entire ring $A$,
which involves identification of elements of $A$ which can be deduced from
one to another 
by a global multiplicative unit, implies a quotient set
$A/\!\!\sim$\index{quotient set}. We say that two elements $a$ and $b$
of $A$ are equivalent  
if there exists a unit $u$ of $A$ such that $b=ua$. An element of
the quotient set $A/\!\!\sim$ is an equivalence
class\index{equivalence class} associated to
an element $a$ of $A$ which is denoted $[a]$ and which is defined by:
$$
[a]=\{b\in A\,:\,\exists\,u\in\,[1],\,b=ua\}
$$
where $[1]$ consists of all the units of $A$. The element $a$ of $A$
is said a representative of its equivalence class.
\\[0.1in]
This quotient set equipped with a suitable
multiplicative law $\cdot$, forms a group\index{group} denoted
$(A/\!\!\sim,\cdot)$. Indeed, defining $[a]\cdot [b]$ as the set of
all the products of any element of $[a]$ and any element of $[b]$, we
have obviously:
$$
[a]\cdot [b]=[ab]
$$ 
Notice that the neutral element of this group is the group of units of $A$
which is identical to the equivalence class $[1]$.
\\[0.1in]
In general, we have:
$$
[a]+[b]=\bigcup_{u\in [1]}[a+ub]
$$
So, from the quotient set $A/\!\!\sim$, regarding the product of sets and
the sum set, it is not possible to form a ring. Indeed, the formula of $[a]+[b]$
above implies that we cannot establish a ring homomorphism\index{ring
homomorphism} between 
$\mathbb{Z}$ and $A/\!\!\sim$ when the product of sets and the sum set are the
operations defined on $A/\!\!\sim$. Or, the rings form a
category\index{category}
and $\mathbb{Z}$ is an initial object in this category\cite{pa}. So,
$A/\!\!\sim$ equipped with the sum set and the product of sets cannot
be a ring.
\\[0.1in] 
The group $(A/\!\!\sim,\cdot)$ is not necessarily isomorphic to
$(\mathbb{Z}/\!\!\sim,\cdot)\cong (\mathbb{N},\cdot)$ when $A$
contains at least a 
unit which is not equal to $\pm 1$. Moreover, the set of arithmetic
properties in 
$A$ which are invariant under $\sim$ may be not necessarily identical
to the set of arithmetic properties in $\mathbb{Z}$ which are invariant under
$\sim$. Accordingly, when the relation of order
defined on an entire ring is not total, the picture of network of elements which
can be put in relation by equivalence relation $\sim$, changes. It may
be possible that an algebra in entire 
ring induces arithmetic which could be different than in
$\mathbb{Z}$. Nevertheless, in this paper, it is shown that  
when a ring extension of an entire subring $A$ of a subfield of 
$\mathbb{C}=\mathbb{R}[i]$ 
is equipped with a magnitude function\footnote{This function
can be viewed as a partial norm. In some cases, this function is a
norm.} which represents the size of the 
elements of this ring extension in a subfield of $\mathbb{R}$, 
the most standard arithmetic properties can be recovered
in $A$ if $A$ is principal,
provided some conditions on the magnitude function are fulfilled.
\\[0.1in]
The plan of the paper is the following one.
In the section \ref{s1}, we recall some basic facts about the group of
units of an entire ring. In the section \ref{s2}, we give some
properties which characterize ideals of a principal entire ring. In
the section \ref{s3}, we define divisibility in an entire ring. In the
section \ref{s4}, we deal with the algebraic structure of a
$\star$-extension of an entire ring equipped with a magnitude
function where the generated set of elements of this ring extension 
is subset of an abelian group. This
magnitude function is a kind of map which preserves multiplicative law
between two algebras and which generalizes 
the concept of a norm defined on a vector space over a field. In
the section \ref{s5}, we give some basic arithmetic 
definitions relative to divisibility in an entire ring. It
leads to get the generalization of the fundamental 
theorem of arithmetic in a principal entire ring. In the section
\ref{s6}, set operations on ideals of a principal entire ring are
connected to divisibility. In the section \ref{s7}, famous arithmetic
theorems as the Bezout identity 
are extended to a principal entire ring. In the section
\ref{s8}, some arithmetic properties on the set of ideals of a
principal entire ring are derived. In the section \ref{s9}, a
maximal ideals of a principal entire ring are obtained from prime
ideals of a principal entire ring. In the section \ref{s10}, examples
of ring extensions of entire rings are developed. In the section
\ref{s11}, an algebra of entire ring generated by the generators of a
Lie algebra is illustrated. It gives a generalization of the concept
developed along this paper of a ring extension of an entire ring with
an abelian group. But, in this algebraic structure, it turns out to be
that the magnitude function is either not defined or degenerate. 

\newpage

\section{{Group of units of an entire ring}}\label{s1}

An entire ring $A$\cite{sg} is\index{entire ring} 
a commutative ring which contains $1$ such that 
$1\neq 0$ and such that there are no zero divisors in $A$. There exist
elements in A which are invertible. They form a multiplicative group denoted
$U(A)$\index{group of units} which is called the group of
units\index{group of units} of $A$  
(for instance, $U(\mathbb{Z})=\{1,-1\}$). Notice that $A$ is not
necessarily a division ring\footnote{A ring $R$ is a division ring if,
and only if, its group of units $U(R)$ contains all the non-zero
elements of $R$.}\index{division ring}. For instance $\mathbb{Z}$ is a
principal entire ring but not a division ring since all the non-zero elements of
$\mathbb{Z}$ are not invertible. Since the multiplicative law
of $A$ is associative and defining
$x^n=\underbrace{xx\ldots x}_{n\,\,\mathrm{times}}$ for all $x\in A$
and for all $n\in\mathbb{N}$ by recurrence:
$$
x^0=1\,\,\,\,\mathrm{and}\,\,\,\,x^{n+1}=xx^n
$$
it can be shown by induction and by regarding inverse of any element $u$ of
$U(A)$ that:
$$
x^{m+n}=x^mx^n\,\,\,\,\mathrm{and}\,\,\,\,(x^m)^n=x^{mn}\,\,\,\,
\forall\,\,x\in A,\,\,\forall\,\,(m,n)\in\mathbb{N}^2
$$
and:
$$
v^{k+l}=v^kv^l\,\,\,\,\mathrm{and}\,\,\,\,(v^{k})^l=v^{kl}\,\,\,\,
\forall\,\,v\in U(A),\,\,\forall\,\,(k,l)\in\mathbb{Z}^2
$$
Notice that since $U(A)$ is a multiplicative group, $v^n$ for all
$v\in U(A)$ and for all $n\in\mathbb{N}$ is invertible and its inverse is
$(v^n)^{-1}=v^{-n}$ for all $v\in U(A)$ and for all $n\in\mathbb{N}$.
\\[0.1in]
If $U(A)$ is a finite group, denoting $|X|$ the order of a
finite subset $X$ of $A$, 
then from the little theorem of Lagrange, we have:
$$
v^{|U(A)|}=1\,\,\,\,\forall\,\,v\in U(A)
$$
with the order $|v|$ of the element $v$ of $A$ which divides $|U(A)|$.

\section{{Ideals and set operations on principal ideals
of an entire ring}}\label{s2}

In this section, $A$ denotes an entire ring. 
We recall that a (left/right) ideal\index{ideal} of a ring
$A$ is an additive subgroup of $A$ which is stable by (left/right)
multiplication. Moreover, a (left/right) ideal of a ring $A$ is
principal if it is generated by a 
singleton $\{a\}$ with $a\in A$. A principal ideal generated
by a singleton $\{a\}$ with $a\in A$, is denoted $aA$. Thus, we have:
$$
aA=\{ax\,:\,x\in A\}
$$

\begin{remark}\label{r2.1}
The intersection of two principal ideals $\mathfrak{a}$ and $\mathfrak{b}$
is an ideal. Let prove that $\mathfrak{a}\cap\mathfrak{b}$ is
really an ideal of $A$. 
Since $\mathfrak{a}$ and $\mathfrak{b}$ are principal
ideals, there exist two 
elements $a,b$ of $A$ such that $\mathfrak{a}=aA$ and
$\mathfrak{b}=bA$.
Since $aA\cap bA$ contains $0=a0=b0$,
since $aA\cap bA$ is stable by addition namely
$\forall (ax_1=by_1,ax_2=by_2)\in (aA\cap bA)^2$ such that $x_i,y_j\in
A$ with $i,j=1,2$, we have  
$ax_1+ax_2=by_1+by_2$ which implies that $a(x_1+x_2)=b(y_1+y_2)\in
aA\cap bA$ and since 
$\forall c\in aA\cap bA$,
$-c\in aA\cap bA$, $aA\cap bA$
is an additive subgroup of $A$. Moreover,
$\forall x\in A$, $\forall c\in aA\cap bA$, we have
$c\in aA$ which 
implies $cx\in aA$ since $aA$ is an ideal and
$c\in bA$ which implies 
that $cx\in bA$ since $bA$ is an ideal. It follows
that $\forall x\in A$, 
$\forall c\in aA\cap bA$, $cx\in aA\cap bA$. So, $aA\cap bA$ is an 
additive subgroup of $A$ which is stable by multiplication. Thus, 
$aA\cap bA$ is an ideal.
\\[0.1in]
The sum of two principal ideals $\mathfrak{a}$ and $\mathfrak{b}$
is an ideal. Let prove that $\mathfrak{a}+\mathfrak{b}$ is
really an ideal of $A$. Since $\mathfrak{a}$ and $\mathfrak{b}$ are principal
ideals, there exist two
elements $a,b$ of $A$ such that $\mathfrak{a}=aA$ and
$\mathfrak{b}=bA$. 
$aA+bA$ contains $0=a0+b0$. Moreover, $\forall (r,s)\in A^2$,
$\forall (ax_1+by_1,ax_2+by_2)\in (aA+bA)^2$ such that $x_i,y_j\in A$
with $i,j=1,2$, we have  
$r(ax_1+by_1)+s(ax_2+by_2)=a(rx_1+sx_2)+b(ry_1+sy_2)\in aA+bA$. Thus, 
$aA+bA$ is an ideal.
\end{remark}

\noindent
A principal ring\index{principal ring} is a ring such that every ideal is
principal.

\begin{lemma}\label{l2.2}
Let $A$ be an entire principal ring. 
Then, any ideal which contains $1$ is equal to $A$.
The ideal $aA$ is equal to $A$ if, and only if, $a\in U(A)$.
\end{lemma}

\begin{proof}
$A$ is itself an ideal which contains $1$. If an ideal contains $1$,
since it is stable by multiplication, then any element $x=1x$ of $A$
belongs to this ideal. So, $A$ is included in this ideal which is itself
included in $A$. It results that this ideal is equal to $A$.
\\[0.1in]
If the ideal $aA=A$, it means that $1\in aA$. So, there exists a non-zero
element $b\in A$ such that $ab=1$. Therefore, $a$ is
invertible. Reciprocally, if $a$ is invertible, there exists a non-zero
element $b$ of $A$ such that $ab=1$. Since $aA$ is an ideal of $A$,
$ab=1\in A$. Therefore, if $a$ is invertible, then $aA=A$. 
\end{proof}

\section{{Divisibility in an entire ring}}\label{s3}

In this section, we assume that the reader has a knowledge of basic
concepts of the number theory. For a review, the reader can be
referred to \cite{tma}. Moreover, we assume that $A$ is an entire ring
which is not necessarily principal.

\begin{definition}\label{d3.1}
An element $a$ of $A$ divides an element $b$ of $A$, what it is
denoted $a|b$, if there exists $c\in A$ such that
$ac=b$. Then the element $a$ of $A$ is said to be a
divisor\index{divisor} of $b$ of $A$ 
and the element $b$ of $A$ is said to be a multiple\index{multiple} of
$a$ of $A$.
\end{definition}

\begin{definition}\label{d3.2}
An element $x$ of $A$ is a divisor of zero\index{divisor of zero} if
$x\neq 0$ and if there 
exists an element $y\neq 0$ of $A$ such that $xy=0$.
\end{definition}

\begin{remark}\label{r3.3}
The set of divisors of an element $a$ of $A$ is denoted
$\mathcal{D}(a)$. We have $U(A)\subseteq \mathcal{D}(a)$ for all $a\in
A$. In particular, $\mathcal{D}(1)=U(A)$. More generally, we have:
$$
\mathcal{D}(a)\subseteq\bigcup_{d|a}dA
$$
If $a\not\in U(A)$, there exists at least an element $d\neq a$ of
$\mathcal{D}(a)$ which is not in $U(A)$ since $U(A)$ is a
multiplicative group. Indeed, let assume absurdly that there doesn't
exist such an element $d\neq a$ of $\mathcal{D}(a)$. Accordingly, all
the elements of $\mathcal{D}(a)$ except $a$ would be in $U(A)$. But, then
$a\in U(A)$ since $U(A)$ 
is a multiplicative group. So, we reach to a
contradiction meaning that if $a\not\in U(A)$, then 
there exists at least an element $d\neq a$ of
$\mathcal{D}(a)$ which does not belong to $U(A)$.
\\[0.1in]
The set of multiples of an element $a$ of $A$ is denoted
$\mathcal{M}(a)$ which is equal to the ideal $aA$:
$$
\mathcal{M}(a)=aA
$$
In particular, we have $\mathcal{D}(0)=\mathcal{M}(0)=0$ and 
$0\in\mathcal{M}(a)$ for all $a\in A$.
\\[0.1in]
The relation of divisibility which is defined on $A$ is reflexive,
transitive and linear\footnote{The property that the relation of
divisibility denoted $|$, which is defined on $A$, is reflexive,
transitive and linear (see also \cite{tma}), means that:\\
reflexivity: $\forall\,\,x\in A$, we have $x|x$;
\\
transitivity: $\forall\,x,y,z\in A$, $x|y$ and $y|z$ imply $x|z$;
\\
linearity: for $x,y,z\in A$ such that $x|y$ and $x|z$, we have
$x|(ay+bz)$ for all $a,b\in A$.}.
\end{remark}

\begin{property}\label{p3.4}
$$
\forall\,(x,y)\in A^2,\,x|y\,\Leftrightarrow\,yA\subseteq xA
$$
\end{property}

\begin{proof}
Indeed, the ideal $xA$ is the set of multiples of $x$. Let assume
$x|y$. Whatever $z\in yA$, we have $y|z$, so by transitivity of the
relation of divisibility defined on $A$, $x|z$ and
$z\in xA$. If, reciprocally, we have $yA\subseteq xA$, it comes that
$y\in xA$, so $x|y$.
\end{proof}
\noindent
In the following, we denote by $aU(A)$, the subset of $A$ defined by:
$$
aU(A)=\{au\,:\,u\in U(A)\}
$$
and for two subsets $X,Y$ of $A$, the subset 
$X\setminus Y$ of $A$ is the subset of all the elements of $X$ which are
not in $Y$:
$$
X\setminus Y=\{x\in X\,:\,x\not\in Y\}
$$
\begin{remark}\label{r3.5}
$$
\mathcal{D}(a)=\bigcup_{d|a}dU(A)
$$
\end{remark}

\section{{Extension of an entire ring}}\label{s4}

\begin{definition}\label{d4.1}
Let $A$ be an entire subring with $1\in A$ ($1\neq 0$), 
of a subfield $\mathbb{F}$ of 
$\mathbb{C}$, which is generated by a finite number of
its elements 
and let $G$ be a finite abelian group of $\mathbb{F}$ which is not contained
in $A$ although the intersection of $A$ and $G$ is
non-empty (since it contains at least $1$) 
and such that all the elements of $G$ commute with all
the elements of $A$. We denote $\mathcal{G}$ a maximal family of linearly
independent elements of $G$ over $\mathbb{F}$\footnote{A maximal family of linearly independent elements of a set $E$ of $\mathbb{C}$ over a ring $R$ of $\mathbb{C}$ is a free family of elements of $E$ over $R$, whose cardinality is maximal. See also the definition \ref{d4.57}.} and $S$ the subset of
$\mathcal{G}$ which consists of all the elements in $\mathcal{G}$ which does
not belong to $A$. 
We denote by $A[S]$ the commutative subring of $\mathbb{F}$ which is
the ring extension of $A$ which includes all linear combinations of
elements of $G$ with coefficients in $A$. It is understood that
$A[S]$ is generated by $S$ but a basis of $A[S]$ is a set of elements which
includes $S$ and a maximal family of linearly independent generators
of $A$ over $A$.
Moreover, we assume that there is no divisor of zero in
$A[S]$. So, since $1\neq 0$ in $A$, $A[S]$ is entire.
Besides, the definition of divisibility in $A$ can be extended in
$A[S]$ with the same notations. 
We assume that $\mathbb{F}$ is equipped with a norm
$||\,\,||_{\mathbb{F}}$. We assume that
any non-zero element of $A[S]$ is invertible in $\mathbb{F}$.
\end{definition}

\begin{remark}\label{r4.2}
$G$ as well as $\mathcal{G}$ and so also $S$ are contained in $U(A[S])$.
\end{remark}

\begin{definition}\label{d4.3}
The magnitude function\index{magnitude function} $N$ is the map defined on
$A[S]$ which associates the unique element $N(x)$ of the set 
$\mathbb{F}\cap\mathbb{R}$, to
$x\in A[S]$: 
\begin{eqnarray}
N:A[S]&\longrightarrow& \mathbb{F}\cap\mathbb{R}
\nonumber\\
x&\mapsto& N(x)
\nonumber
\end{eqnarray}
with properties:
\[
\forall\,\,x\in A[S]\setminus (\ker N\setminus\{0\})\,:\,N(x)\in A[S],\,\,
\mathcal{D}(N(x))=\mathcal{D}(x)\label{p.1}\tag{P.1}
\\[0.05in]
\]
\[
N(x)=0\,\Leftrightarrow\,x\in\ker N\label{p.2}\tag{P.2}
\\[0.05in]
\]
\[
\forall\,\,x,y\in A[S],\,\,\exists\,\,z\in A[S]\,:\,N(x)+N(y)=N(z)
\label{p.3}\tag{P.3}
\\[0.05in]
\]
\[
N(x)=N(-x)\,\,\,\,\forall\,\,x\in A[S]\label{p.4}\tag{P.4}\\[0.05in]
\]
\[
\forall\,\,x\in A[S]\,:\,N(x)\in A[S],\,\,
N(N(x))=N(x)\label{p.5}\tag{P.5}\\[0.05in]
\]
\[
N(xy)=N(x)N(y)\,\,\,\,\forall\,\,x,y\in A[S]\label{p.6}\tag{P.6}\\[0.05in]
\]
\[
||N(x)||_{\mathbb{F}}=N(x),\,\,\,\forall\,\,x\in A[S]
\label{p.7}\tag{P.7}\\[0.05in]
\]
\[
\forall\,\,x\in A[S]\setminus\ker N,
\,\exists\,\,x'\in\mathbb{F}\cap\mathbb{R}\,:\,N(x)^2=xx'\in A[S]
\label{p.8}\tag{P.8}\\[0.05in]
\]
For a given element $x$ of $A[S]$, $N(x)$ is said to be the size or 
the magnitude\index{magnitude} of $x\in A[S]$ in $\mathbb{F}\cap\mathbb{R}$.
\end{definition}

\begin{remark}\label{r4.4}
For a given element $x$ of $A[S]$, 
the notation $N(x)^{-1}$ means the inverse of $N(x)$ in
$\mathbb{F}\cap\mathbb{R}$. It is possible that $N(x)$ belongs to
$U(A[S])$. In such a case, $N(x)^{-1}$ is the inverse of $N(x)$ in
$A[S]$. For instance (see below for more details), when $x\in U(A[S])$, we have
$N(x)\in U(A[S])$ and $N(x)^{-1}=N(x^{-1})$.
\end{remark}

\begin{example}\label{e4.5}
For any integer $k$ of $\mathbb{Z}$,
$N(k)$ is equal to the unsigned part of $k$. In other
word, we have:
$$
N(k)=\mathrm{abs}(k)\,\,\,\,\forall\,\,k\in\mathbb{Z}
$$
where $\mathrm{abs}$ is the absolute value (or modulus) function on
$\mathbb{Z}$. 
\end{example}

\begin{property}\label{p4.6}
Let $n\in\mathbb{N}^{\star}$ and let $x_1,\ldots,x_n$ be $n$
element(s) of $A[S]$. Then we have:
$$
N(x_1)+\ldots+N(x_n)=0\,\Leftrightarrow\,x_1,\ldots,x_n\in\ker N
$$
\end{property}

\begin{proof} 
From the property (\ref{p.2}) of the definition (\ref{d4.3}), we have
$N(x_1)=0\,\Leftrightarrow\,x_1\in\ker N$. So, the property is verified for
$n=1$. In the following, we assume that $n\geq 2$.
\\[0.1in]
From the property (\ref{p.2}) of the definition (\ref{d4.3}), 
it is obvious also that if $x_1,\ldots,x_n\in\ker N$, then
$N(x_1)+\ldots+N(x_n)=0$.
\\[0.1in]
Reciprocally, if $N(x_1)+\ldots+N(x_n)=0$ with $n\geq 2$, we have
($n\geq 2$): 
$$
N(x_2)+\ldots+N(x_n)=-N(x_1)
$$
From the property (\ref{p.3}) of the definition (\ref{d4.3}), using an
immediate reasoning by induction, we can
find an element $z$ of $A[S]$ such that ($n\geq 2$):
$$
N(x_2)+\ldots+N(x_n)=N(z)
$$
It results that:
$$
N(z)=-N(x_1)
$$
Using the property (\ref{p.7}) of the definition
(\ref{d4.3}), we have:
$$
||N(z)||_{\mathbb{F}}=N(z)
$$
Using again (\ref{p.7}), from the equality $N(z)=-N(x_1)$, it implies that:
$$
N(z)=||-N(x_1)||_{\mathbb{F}}=||N(x_1)||_{\mathbb{F}}=N(x_1)
$$
So:
$$
N(x_1)=-N(x_1)\,\Leftrightarrow\,N(x_1)=0\,\Leftrightarrow\,x_1\in\ker N
$$
Since we can exchange $x_1$ with any $x_i$ for $i\in\llbracket
1,n\rrbracket$, we deduce that if $N(x_1)+\ldots+N(x_n)=0$ with
$n\geq 2$, then $x_1,\ldots,x_n\in\ker N$.
\\[0.1in]
We conclude that if $x_1,\ldots,x_n$ be $n$
element(s) of $A[S]$ with $n\in\mathbb{N}^{\star}$,
$N(x_1)+\ldots+N(x_n)=0$ is equivalent to $x_1,\ldots,x_n\in\ker N$
with $n\in\mathbb{N}^{\star}$.
\end{proof}

\begin{definition}\label{d4.7}
An element $r$ of a ring $R$ is said regular\index{regular} or
simplifiable for the 
multiplicative law of $R$ (or multiplicatively regular or
multiplicatively simplifiable) if for any couple $(x,y)\in R^2$, we have:
$$
rx=ry\,\Rightarrow\,x=y
$$
and
$$
xr=yr\,\Rightarrow\,x=y
$$
\end{definition}

\begin{remark}\label{r4.8}
Since $A[S]$ is a commutative ring, any element $a$ is multiplicatively
regular if for any couple $(x,y)\in A[S]^2$, we have:
$$
ax=ay\,\Rightarrow\,x=y
$$
Any invertible element of $A[S]$ is regular for the multiplicative law of
$A$ (or multiplicatively regular). Indeed, let $a$ be an invertible
element of $A[S]$, whose inverse is $b$. Then, we have:
$$
ax=ay\,\Rightarrow\,bax=bay\,\Rightarrow\,x=y
$$
\end{remark}

\begin{property}\label{p4.9}
A non-zero element of $A[S]$ is regular for the multiplicative law of $A[S]$
(or multiplicatively regular) if, and only if, it is not a divisor of
$0$.
\end{property}

\begin{proof}
Let $a$ a non-zero element of $A[S]$ and let $x,y\in A[S]$.
\\[0.1in]
If $a$ is multiplicatively regular and if $ax=0$, then $ax=a0$. After
simplification by $a$, it comes that $x=0$. So (see the definition
(\ref{d3.2})), $a$ is not a divisor of $0$.
\\[0.1in]
Reciprocally, if $a$ is not a divisor of $0$ and if $ax=ay$, then
$a(x-y)=0$. Necessarily, we have $x-y=0$ and so $x=y$ implying that $a$
is regular. 
\end{proof}

\begin{remark}\label{r4.10}
Since $A[S]$ is an entire ring, $A[S]$ does not contain divisor of
zero. So, any non-zero element of $A[S]$ is regular. In particular, any
prime element of $A$ is regular.
\\[0.1in]
Moreover, in a subfield of $\mathbb{C}$ as for instance $\mathbb{F}$,
$\mathbb{R}$ or $\mathbb{F}\cap\mathbb{R}$, its non-zero elements are
invertible and so regular. 
\end{remark}

\begin{property}\label{p4.11}
$$
N(1)=1
$$
\end{property}

\begin{proof}
Using the properties (\ref{p.4}) and (\ref{p.6}) of the definition
(\ref{d4.3}), we have: 
$$
N(-x)=N(x)\,\,\,\,\forall\,\,x\in A[S]
$$
$$
N((-1)x)=N(x)\,\,\,\,\forall\,\,x\in A[S]
$$
$$
N(-1)N(x)=N(x)\,\,\,\,\forall\,\,x\in A[S]
$$
$$
N(1)N(x)=N(x)\,\,\,\,\forall\,\,x\in A[S]
$$
$$
N(x)(N(1)-1)=0\,\,\,\,\forall\,\,x\in A[S]
$$
In particular, it is true when $x\not\in\ker N$ and so when $N(x)$ 
is regular in $\mathbb{F}\cap\mathbb{R}$ (see the remark
(\ref{r4.10})). Therefore, we get: 
$$
N(1)-1=0
$$
$$
N(1)=1
$$
\end{proof}

\begin{property}\label{p4.12}
$$
0\in\ker N
$$
\end{property}

\begin{proof}
Using the property (\ref{p.6}) of the definition (\ref{d4.3}), we have:
$$
N(0x)=N(0)N(x)\,\,\,\,\forall\,\,x\in A[S]
$$
Since $0x=0$, it comes that:
$$
N(0)=N(0)N(x)\,\,\,\,\forall\,\,x\in A[S]
$$
$$
N(0)(N(x)-1)=0\,\,\,\,\forall\,\,x\in A[S]
$$
In particular, it is true for $x\in A[S]$ such that $N(x)-1\neq
0$. In this case, using the remark
(\ref{r4.10}), $N(x)-1$ is regular in $\mathbb{F}\cap\mathbb{R}$ and
using the property (\ref{p.2}) of the definition (\ref{d4.3}), we get:
$$
N(0)=0\,\Leftrightarrow\,0\in\ker N
$$ 
\end{proof}

\begin{property}\label{p4.13}
Let $v\in U(A[S])$. Then, the inverse of $N(v)$ denoted
$N(v)^{-1}$ in $\mathbb{F}\cap\mathbb{R}$ is:
$$
N(v)^{-1}=N(v^{-1})
$$
\end{property}

\begin{proof}
Using the property (\ref{p.6}) of the definition
(\ref{d4.3}) and the property (\ref{p4.11}), we can see that: 
$$
N(vv^{-1})=N(v)N(v^{-1})=N(1)=1\,\,\,\,\forall\,\,v\in U(A[S])
$$ 
What it proves that $N(v)\in U(A[S])$ for all $v\in A[S]$ 
and the inverse of $N(v)$ denoted $N(v)^{-1}$ in
$\mathbb{F}\cap\mathbb{R}$ is:
$$
N(v)^{-1}=N(v^{-1})\,\,\,\,\forall\,\,v\in U(A[S]) 
$$
\end{proof}

\begin{remark}\label{r4.14}
Accordingly, the restriction of the application $N$ denoted
$N\left|_{U(A[S])}\right.$ is an endomorphism of the
multiplicative group $U(A[S])$.
\end{remark}

\begin{corollary}\label{c4.15}
If $v\in U(A[S])$, then $v\not\in\ker N$.
\end{corollary}

\begin{proof}
Since $N(v)N(v)^{-1}=1$ for all $v\in U(A[S])$ and since $1\neq 0$, we have
$N(v)\neq 0$ for all $v\in U(A[S])$. Therefore, if $v\in U(A[S])$,
then $v\not\in\ker N$.
\end{proof}

\begin{property}\label{p4.16}
$$
N(x)^n=N(x^n)\,\,\,\forall x\in A[S],\,\,\forall n\in\mathbb{N}
$$ 
\end{property}

\begin{proof}
This property is proved by induction by using the property (\ref{p.6})
of the definition (\ref{d4.3}).
\end{proof}

\noindent
Then, we have:
$$
N(v)^{-n}=(N(v)^n)^{-1}=N(v^n)^{-1}=N((v^n)^{-1})
=N(v^{-n})\,\,\,\,\forall v\in U(A[S]),\,\,\forall n\in\mathbb{N} 
$$
It follows the property:

\begin{property}\label{p4.17}
$$
N(v)^k=N(v^k)\,\,\,\,\forall v\in U(A[S]),\,\,\forall k\in\mathbb{Z}
$$
\end{property}

\begin{property}\label{p4.18}
Let $x,y$ be two non-zero elements of $A[S]$. Then:
$$
\mathcal{D}(x)=\mathcal{D}(y)\,\Leftrightarrow\,\exists\,\,v\in U(A[S])\,:\,y=xv
$$
with $v$ which is unique.
\end{property}

\begin{proof}
Let $x,y$ be two non-zero elements of $A[S]$.
\\[0.1in] 
From the remark (\ref{r3.5}), since $U(A[S])$ is a multiplicative group, 
if there exists $v\in U(A[S])$ such that $y=xv$, then
$\mathcal{D}(x)=\mathcal{D}(y)$. Indeed, we have:
$$
\mathcal{D}(y)=\bigcup_{d|y}dU(A[S])=\bigcup_{d|xv}dU(A[S])
=\bigcup_{dv^{-1}|x}dU(A[S])=\bigcup_{d'|x}d'vU(A[S])=\bigcup_{d'|x}d'U(A[S])
=\mathcal{D}(x) 
$$ 
with $d'=dv^{-1}$.
\\[0.1in]
Reciprocally, if $\mathcal{D}(x)=\mathcal{D}(y)$, then $x|y$ and
$y|x$. So, there exist two elements $d,d'$ of $A[S]$ such that $y=dx$ and
$x=d'y$. It gives:
$$
y=dd'y
$$
$$
y(dd'-1)=0
$$
Since $y\neq 0$, from the remark (\ref{r4.10}), $y$ is regular and we
deduce that:
$$
dd'=1
$$
So, $d\in U(A[S])$ and its inverse is $d^{-1}=d'$.
\\[0.1in]
Let consider two elements $v,v'$ of $U(A[S])$ such that ($x,y\neq 0$):
$$
y=xv=xv'
$$
Then we have ($x\neq 0$):
$$
xv-xv'=0
$$
$$
x(v-v')=0
$$
Since $x\neq 0$ of $A[S]$ is regular (see the remark (\ref{r4.10})), we
have:
$$
v-v'=0
$$ 
$$
v=v'
$$
We conclude that 
$\mathcal{D}(x)=\mathcal{D}(y)$ with $x,y\neq 0$ of $A[S]$ 
if, and only if, there exists a unique
element $v$ of $U(A[S])$ such that $y=xv$ with $x,y\neq 0$.
\end{proof}

\begin{corollary}\label{c4.19}
$$
\forall\,\,x\in A[S]\setminus\ker N\,:\,
N(x)\in A[S],\,\,\exists\,\,v\in U(A[S])\,:\,x=vN(x)
$$
with $v$ which is unique.
\end{corollary}

\begin{proof}
Using the property (\ref{p.1}) of the definition
(\ref{d4.3}), the corollary (\ref{c4.19}) follows from the property
(\ref{p4.18}). 
\end{proof}

\begin{property}\label{p4.20}
Let $x$ be an element of $A[S]\setminus(\ker N\setminus\{0\})$ and let
$n\in\mathbb{N}$. If $N(x)\in A[S]$ and if $N(nx)\in A[S]$, then
we have ($x\not\in\ker N\setminus\{0\}$):  
$$
N(nx)=nN(x)
$$
\end{property}

\begin{proof}
Let $x$ be an element of $A[S]\setminus (\ker N\setminus\{0\})$ and let
$n\in\mathbb{N}$.
\\[0.1in]
The property is verified for $n=0$ and for $x=0$ 
since $0x=n0=0$ and $N(0)=0$. In the following,
we assume that $n\in\mathbb{N}^{\star}$ and $x\neq 0$.
\\[0.1in]
Notice that the element $nx$ for $x\in A[S]$ and $n\in\mathbb{N}^{\star}$
is well defined in $A[S]$ since 
$nx=\underbrace{x+\ldots+x}_{n\,\,\mathrm{times}}$ and $(A[S],+)$ is an
additive subgroup of $A[S]$.
Moreover, from the property (\ref{p.1}) of the definition
(\ref{d4.3}), if $N(x)\in A[S]$, then we have ($x\not\in\ker N$):
$$
\mathcal{D}(x)=\mathcal{D}(N(x))
$$
It means that ($d\in A[S]$):
$$
d|x\,\Leftrightarrow\,d|N(x)
$$
We know that ($d\in A[S]$):
$$
d|x\,\Rightarrow\,nd|nx
$$
Or, since $n\in\mathbb{N}^{\star}$ and so 
$n1=\underbrace{1+\ldots+1}_{n\,\,\mathrm{times}}$ is regular (see the
remark (\ref{r4.10})), we have ($d\in A[S]$):
$$
nd|nx\,\Rightarrow\,d|x
$$
It follows that ($d\in A[S]$ and $n\in\mathbb{N}^{\star}$):
$$
d|x\,\Leftrightarrow\,nd|nx
$$
and we have also ($d\in A[S]$ and $n\in\mathbb{N}^{\star}$):
$$
d|N_A(x)\,\Leftrightarrow\,nd|nN(x)
$$
So ($d\in A[S]$ and $n\in\mathbb{N}^{\star}$):
$$
nd|nx\,\Leftrightarrow\,nd|nN(x)
$$
$$
d'|nx\,\Leftrightarrow\,d'|nN(x)
$$
with $d'=nd$. Consequently, renaming $d'$ as $d$, we have ($d\in A[S]$
and $n\in\mathbb{N}^{\star}$):
$$
d|nx\,\Leftrightarrow\,d|nN(x)
$$
Therefore, using the remark (\ref{r3.5}), we have ($x\not\in\ker N$ and
$N(x)\in A[S]$):
$$
\mathcal{D}(nx)=\bigcup_{d|nx}dU(A)=\bigcup_{d|nN_A(x)}dU(A)
=\mathcal{D}(nN(x))
$$
Again, from the property (\ref{p.1}) of the definition
(\ref{d4.3}), if $N(nx)\in A[S]$, then we have ($x\not\in\ker N$): 
$$
\mathcal{D}(nx)=\mathcal{D}(N(nx))
$$
It results that ($x\not\in\ker N$, $N(x)\in A[S]$
and $N(nx)\in A[S]$): 
$$
\mathcal{D}(N(nx))=\mathcal{D}(nN(x))
$$
Or, from the property (\ref{p.3}) of the definition
(\ref{d4.3}), using an immediate reasoning by induction, there exists
an element $z\in A[S]$ such that $nN(x)=N(z)$. So, using the property
(\ref{p.5}) of the definition (\ref{d4.3}), we have
$N(nN(x))=nN(x)$. Therefore since $N(nx)$ is the unique
element of $\mathcal{D}(nx)$ such that $N(N(nx))=N(nx)$ (see the
definition (\ref{d4.3})) when $N(nx)\in A[S]$, we get $N(nx)=nN(x)$. 
\end{proof}

\begin{corollary}\label{c4.21}
Let $x$ be an element of $A[S]\setminus (\ker N\setminus\{0\})$ and let
$k\in\mathbb{Z}$. Then we have: 
$$
N(kx)=\mathrm{abs}(k)N(x)
$$
\end{corollary}

\noindent
If $U(A[S])$ has a finite order, since $\forall v\in U(A[S])$,
$v^{|U(A[S])|}=1$, we have $N(v)^{|U(A[S])|}=1$. It comes that:
$$
N(v)^{|U(A[S])|}-1=0\,\,\,\,\forall\,\,v\in U(A[S])
$$
$$
(N(v)-1)\left\{1+\ldots+N(v)^{|U(A[S])|-1}\right\}=0\,\,\,\,
\forall\,\,v\in U(A[S])
$$
Since
$N(x_1)+\ldots+N(x_n)=0\,\Leftrightarrow\,x_1,\ldots,x_n\in\ker N$
with $n\in\mathbb{N}^{\star}$ (see the property (\ref{p4.6}) above)
and since $1\neq 0$, the factor $1+\ldots+N(v)^{|U(A[S])|-1}$ is
regular and so:
$$
N(v)-1=0\,\,\,\,\forall\,\,v\in U(A[S]) 
$$
$$
N(v)=1\,\,\,\,\forall\,\,v\in U(A[S]) 
$$
Therefore, using also the property (\ref{p.6}) of the definition
(\ref{d4.3}), we have:

\begin{corollary}\label{c4.22}
If $U(A[S])$ has a finite order, then we have:
$$
N(v)=1\,\,\,\,\forall\,\,v\in U(A[S]) 
$$
$$
N(xv)=N(x)\,\,\,\,\forall\,\,x\in A[S],\,\,\forall\,\,v\in U(A[S]) 
$$
\end{corollary}

\begin{remark}\label{r4.23}
If $U(A[S])$ has a finite order, then the endomorphism
$N\left|_{U(A[S])}\right.$ is surjective such 
that $N\left|_{U(A[S])}\right(U(A[S]))=\{1\}$.
\end{remark}

\begin{property}\label{p4.24}
Let $x$ be an element of $A[S]$. Then we have:
$$
N(x)=1\,\Rightarrow\,x\in U(A[S])
$$
\end{property}

\begin{proof}
Let $x$ be an element of $A[S]$.
\\[0.1in]
If $N(x)=1$, then from the property (\ref{p.2}), $x\not\in\ker
N$. Using the property (\ref{p.8}) of the definition (\ref{d4.3}), 
there exists an element $x'\not\in\ker N$ of $A[S]$ such that
$N(x)^2=xx'$. Since $N(x)=1$, it implies that $xx'=1$. So, $x$ is
invertible in $A[S]$ and its inverse is $x'$. Therefore, if
$N(x)=1$, then $x\in U(A[S])$.
\end{proof}

\noindent
We saw that an element $x$ of $A[S]$ and its magnitude $N(x)$ differ
from each other by an unit which is 
unique, when $x\not\in\ker N$ such that $N(x)\in A[S]$ (see the
corollary (\ref{c4.19})). This result is extended in 
the following definition (\ref{d4.25}).

\begin{definition}\label{d4.25}
The unit function\index{unit function} is the map defined on
$A[S]\setminus\ker N$ which associates the 
unique element $u(x)$ of $\mathbb{F}$ to $x\not\in\ker N$ of $A[S]$:
\begin{eqnarray}
u:A[S]\setminus\ker N&\rightarrow&\mathbb{F}
\nonumber\\
x&\mapsto& u(x)
\nonumber
\end{eqnarray}
with properties:
\[
x=u(x)N(x)\,\,\,\,\forall\,\,x\in A[S]\setminus\ker N
\label{p.9}\tag{P.9}\\[0.05in]
\]
\[
\forall\,\,x\in A[S]\setminus\ker N\,:\,u(x)\in A[S]\setminus\ker N,\,\,
u(u(x))=u(x)\label{p.10}\tag{P.10}
\]
For given $x\not\in\ker N$ of $A[S]$, $u(x)$ is said to be 
the unit part\index{unit part} of $x$.
\end{definition}

\begin{remark}\label{r4.26}
For a given element $x$ of $A[S]\setminus\ker N$, 
the notation $u(x)^{-1}$ means the inverse of $u(x)$ in
$\mathbb{F}$. It is possible that $u(x)$ belongs to
$U(A[S])$. In such a case, $u(x)^{-1}$ is the inverse of $u(x)$ in
$A[S]$. The theorem (\ref{t4.38})
allows to make the expression $u(x)^{-1}$ definite.
\end{remark}

\begin{property}\label{p4.27}
Let $x$ be an element of $A[S]\setminus\ker N$. Then ($x\not\in\ker N$):
$$
u(x)+u(-x)=0
$$
\end{property}

\begin{proof}
From the property (\ref{p.4}) of the definition (\ref{d4.3}), using
the definition (\ref{d4.25}), we have:
$$
-x=-u(x)N(x)=-u(x)N(-x)
$$
Or, from the property (\ref{p.9}) of the definition (\ref{d4.25}), we have:
$$
-x=u(-x)N(-x)
$$
So:
$$
u(-x)N(-x)=-u(x)N(-x)
$$
Since $N(x)=N(-x)$ is regular for all $x\not\in\ker N$ of $A[S]$ (see the
remark (\ref{r4.10})), it results
that:
$$
u(-x)=-u(x)
$$
\end{proof}

\begin{property}\label{p4.28}
Let $x,y$ be two elements of $A[S]\setminus\ker N$. 
Then ($x,y\not\in\ker N$):
$$
u(xy)=u(x)u(y)
$$
\end{property}

\begin{proof}
Let $x,y$ be two elements of $A[S]\setminus\ker N$. From the property
(\ref{p.6}) of the definition (\ref{d4.3}), we have ($x,y\not\in\ker N$):
$$
xy=u(xy)N(xy)=u(xy)N(x)N(y)
$$
Or, since $x=u(x)N(x)$ and $y=u(y)N(y)$, it comes that
($x,y\not\in\ker N$):
$$
xy=u(x)N(x)u(y)N(y)=u(x)u(y)N(x)N(y)
$$
So:
$$
u(xy)N(x)N(y)=u(x)u(y)N(x)N(y)
$$
Since $x,y$ are non-zero and so regular (see the remark (\ref{r4.10}))
as well as $N(x),N(y)$, it results that ($x,y\not\in\ker N$):
$$
u(xy)=u(x)u(y)
$$
\end{proof}

\begin{property}\label{p4.29}
Let $x$ be an element of $A[S]$. 
\\[0.1in]
If $x\not\in\ker N$ and $u(x)\in A[S]$, then we have
($x\not\in\ker N$):
$$
N(u(x))=1
$$
If $N(x)\in A[S]\setminus\ker N$, then we have ($N(x)\not\in\ker N$):
$$
u(N(x))=1
$$
\end{property}

\begin{proof}
Let $x$ be an element of $A[S]$.
\\[0.1in] 
If $x\not\in\ker N$ and $u(x)\in A[S]$, then 
from the properties (\ref{p.5}), (\ref{p.6}) of the definition
(\ref{d4.3}) and from 
the property (\ref{p.9}) of the definition (\ref{d4.25}), we have
($x\not\in\ker N$ and $u(x)\in A[S]$): 
$$
N(x)=N(u(x)N(x))=N(u(x))N(N(x))=N(u(x))N(x)
$$
$$
N(x)(N(u(x))-1)=0
$$
Since $x\not\in\ker N$, from the property (\ref{p.2}) of the definition
(\ref{d4.3}), $N(x)\neq 0$. So, since the set
$A[S]$ does not
contain divisor of $0$, we get ($x\not\in\ker N$ and $u(x)\in A[S]$):
$$
N(u(x))-1=0
$$
$$
N(u(x))=1
$$
Moreover, if $N(x)\in A[S]\setminus\ker N$, then 
from the properties (\ref{p.9}) and (\ref{p.10}) of the definition
(\ref{d4.25}) of the unit  
part of an element $x\not\in\ker N$ of $A$, using the property
(\ref{p4.28}), we 
have ($N(x)\in A[S]\setminus\ker N$):
$$
u(x)=u(u(x)N(x))=u(u(x))u(N(x))
$$
$$
u(x)=u(x)u(N(x))
$$
$$
u(x)(u(N(x))-1)=0
$$
Since $u(x)$ for $x\not\in\ker N$ is well defined and is regular
($u(x)$ is invertible in $\mathbb{F}$), we have 
($N(x)\in A[S]\setminus\ker N$): 
$$
u(N(x))-1=0
$$
$$
u(N(x))=1
$$
Therefore, we conclude that if $u(x)\in A[S]$ then ($x\not\in\ker N$):
$$
N(u(x))=1
$$
and if $N(x)\in A[S]\setminus\ker N$ then ($N(x)\not\in\ker N$):
$$
u(N(x))=1
$$
\end{proof}

\begin{remark}\label{r4.30}
The equality $N(u(x))=1$ for all $x\in A[S]\setminus\ker N$ such
that $u(x)\in A[S]$, is true even if $U(A[S])$ has not finite order.
\end{remark}

\begin{property}\label{p4.31}
Let $x$ be an element of $A[S]\setminus\ker N$. Then we have
($x\not\in\ker N$):
$$
N(x)=x\,\Leftrightarrow\,u(x)=1
$$
\end{property}

\begin{proof}
Let $x$ be an element of $A[S]\setminus\ker N$. 
\\[0.1in]
If $u_A(x)=1$, then from the property (\ref{p.9}) of the definition
(\ref{d4.25}) of the unit part 
of an element $x\not\in\ker N$ of $A[S]$, it is obvious that $x=N(x)$.
\\[0.1in]
Reciprocally, if $N(x)=x$, from the property (\ref{p4.29}), we have
($x\not\in\ker N$): 
$$
u(N(x))=u(x)=1
$$
Therefore, we conlude that for any $x\not\in\ker N$ of $A[S]$, 
we have $N(x)=x\,\Leftrightarrow\,u(x)=1$.
\end{proof}

\begin{example}\label{e4.32}
We know that $N(1)=1$. So, we have, $u(1)=1$.
\end{example}

\begin{property}\label{p4.33}
Let $x$ be a non-zero element of $A[S]\setminus\ker N$. Then we have
($x\not\in\ker N$): 
$$
N(x)=-x\,\Leftrightarrow\,u(x)=-1
$$
\end{property}

\begin{proof}
Let $x$ be an element of $A[S]\setminus\ker N$. 
\\[0.1in]
If $u(x)=-1$, then from the property (\ref{p.9}) of the definition
(\ref{d4.25}) of the unit part of an element
$x\not\in\ker N$ of $A[S]$, it is obvious that $x=-N(x)$.
\\[0.1in]
Reciprocally, if $N(x)=-x$, from the properties (\ref{p4.27}) and
(\ref{p4.29}), we have ($x\not\in\ker N$):
$$
u(N(x))=u(-x)=-u(x)=1
$$
Therefore, we conlude that for any $x\not\in\ker N$ of $A[S]$, 
we have $N(x)=-x\,\Leftrightarrow\,u(x)=-1$.
\end{proof}

\begin{example}\label{e4.34}
We know that $N(-1)=N(1)=-(-1)$. So, we have, $u(-1)=-1$.
\end{example}

\begin{remark}\label{r4.35}
If $x\in U(A[S])$ (as for instance $x\in G\supset S$), then $x^{-1}$
exists and is unique in $A[S]$ and since $x\not\in\ker N$ in this
case, from the property (\ref{p.8}) of the definition (\ref{d4.3}),
there exists an element $x'$ 
of $\mathbb{F}\cap\mathbb{R}$ such
that $N(x)^2=xx'\in A[S]$. So, if $x\in U(A[S])$, since $A[S]$ is a
ring, then $x'=x^{-1}N(x)^2$ exists and belongs to
$A[S]$. It follows that ($x\in U(A[S])$ and so $x'\in A[S]$):
$$
N(x')=N(x)\neq 0
$$
Thus, if $x\in U(A[S])$, then $x'\not\in\ker N$. 
\\[0.1in]
Notice that $x'$ is unique. Indeed, let $x''\in A[S]\setminus\ker N$ such
that ($x,x',x''\not\in\ker N$):
$$
xx'=xx''
$$
Then since $x\not\in\ker N$ and so since $x$ is regular, we get
($x',x''\not\in\ker N$):
$$
x'=x''
$$
Moreover, since $x,x'\in A[S]\setminus\ker N$, we have
($x\in U(A[S])$ and so $x,x'\not\in\ker N$):
$$
N(x)^2=u(x)N(x)u(x')N(x')=N(x)u(x)u(x')N(x')
$$
$$
N(x)(N(x)-u(x)u(x')N(x'))=0
$$
Since $N(x)$ is regular for $x\not\in\ker N$ of $A[S]$ (see the
remark (\ref{r4.10})), it comes that ($x\in U(A[S])$ and $x,x'\not\in\ker N$):
$$
N(x)-u(x)u(x')N(x')=0
$$
$$
N(x)=u(x)u(x')N(x')
$$
or equivalently ($x\in U(A[S])$ and so $x,x'\not\in\ker N$):
$$
u(x)^{-1}N(x)=u(x')N(x')=x'
$$
We shall identify $x'$ when $x\in U(A[S])$, with what we call
$x^{\star}$ (see below the definition (\ref{d4.36})).
\end{remark}

\begin{definition}\label{d4.36}
The operation $\star$\index{star operation} is the involution 
defined on $A[S]$, 
which maps element $x$ of $A[S]$ 
to element $x^{\star}$ of $A[S]$:
\begin{eqnarray}
\star:A[S]&\longrightarrow& A[S]
\nonumber\\
x&\mapsto& x^{\star}=\left\{\begin{array}{ccc}
x & \mathrm{if} & x\not\in U(A[S])
\\
x' & \mathrm{if} & x\in U(A[S])
\end{array}
\right.
\nonumber
\end{eqnarray}
with $x'$ defined in the remark (\ref{r4.35}) and properties:
\[
(x^{\star})^{\star}=x\,\,\,\,\forall\,\,s\in A[S]
\label{p.11}\tag{P.11}\\[0.05in]
\]
\[
(x+y)^{\star}=x^{\star}+y^{\star}\,\,\,\,\forall\,\,x,y\in A[S]
\label{p.12}\tag{P.12}\\[0.05in]
\]
\[
(xy)^{\star}=y^{\star}x^{\star}\,\,\,\,\forall\,\,x,y\in A[S]
\label{p.13}\tag{P.13}\\[0.05in]
\]
The element $x^{\star}$ of $A[S]$ is said to be the
conjugate\index{conjugate} of the element $x$ of $A[S]$.
\end{definition}

\begin{remark}\label{r4.37}
$$
0^{\star}=0
$$
$$
1^{\star}=1
$$
$$
(-1)^{\star}=-1
$$
For instance, let prove that $0^{\star}=0$. Indeed, using the property
(\ref{p.12}) of the definition (\ref{d4.36}), we have ($x\in A[S]$):
$$
(x+0)^{\star}=x^{\star}+0^{\star}
$$
Or, $x+0=x$. It comes that ($x\in A[S]$):
$$
x^{\star}=x^{\star}+0^{\star}
$$
$$
x^{\star}+0=x^{\star}+0^{\star}
$$
$$
0=0^{\star}
$$
Moreover, let prove that $1^{\star}=1$. Indeed, using the property
(\ref{p.13}) of the definition (\ref{d4.36}), we have ($x\in A[S]$):
$$
(x1)^{\star}=1^{\star}x^{\star}
$$
Since $x1=1x=x$, we obtain ($x\in A[S]$):
$$
x^{\star}=1^{\star}x^{\star}
$$
$$
1x^{\star}=1^{\star}x^{\star}
$$
In particular, it is true when $x$ is regular meaning that:
$$
1=1^{\star}
$$
\end{remark}

\begin{theorem}\label{t4.38}
Let $x$ be an element of $A[S]$. Then we have ($x\in A[S]$):
$$
N(x^{\star})=N(x)
$$
Moreover, if $x\in U(A[S])$, then we have:
$$
u(x^{\star})=u(x)^{-1}
$$
\end{theorem}

\begin{proof}
If $x\not\in U(A[S])$, from the definition (\ref{d4.36}), we have
$x=x^{\star}$. It implies that ($x\not\in U(A[S])$):
$$
N(x)=N(x^{\star})
$$
If $x\in U(A[S])$, then from the definition (\ref{d4.36}) we have
$x^{\star}=x'$ and from the remark (\ref{r4.35}), we know that ($x\in U(A[S])$):
$$
N(x)=N(x')=N(x^{\star})
$$ 
Moreover, if $x\in U(A[S])$, then 
from the remark (\ref{r4.35}), we know that:
$$
u(x)^{-1}N(x)=u(x')N(x')=x'
$$
Using $N(x)=N(x')$, it comes that:
$$
u(x)^{-1}N(x)=u(x')N(x)
$$
Since $x'\not\in\ker N$, $N(x)$ is regular. It results that ($x\in
U(A[S])$):
$$
u(x')=u(x)^{-1}
$$
Therefore, if $x\in U(A[S])$, then we have:
$$
u(x^{\star})=u(x)^{-1}
$$
\end{proof}

\begin{property}\label{p4.39}
Let $x$ be an element of $A[S]\setminus\ker N$. Then we have
($x\not\in\ker N$):
$$
x^{\star}=x\,\Leftrightarrow\,u(x)\in\{1,-1\}
$$
\end{property}

\begin{proof}
Let $x$ be an element of $A[S]\setminus\ker N_A$. 
\\[0.1in] 
Using the definition (\ref{d4.25}) of the unit part of a non-zero element $x$ of
$A[S]$, we have ($x\not\in\ker N_A$):
$$
x^{\star}=x\,\Leftrightarrow\,
u(x^{\star})N(x^{\star})=u(x)N(x)
\,\Leftrightarrow\,
u(x^{\star})N(x)=u(x)N(x)
$$
Since for $x\not\in\ker N$, $N(x)$ is regular (see the remark 
(\ref{r4.10})), we have ($x\not\in\ker N$):
$$
x^{\star}=x\,\Leftrightarrow\,u(x^{\star})=u(x)\,\Leftrightarrow\,
u(x)^{-1}=u(x)\,\Leftrightarrow\,u(x)^2=1
$$
$$
x^{\star}=x\,\Leftrightarrow\,u(x)^2-1=0\,\Leftrightarrow\,
(u(x)-1)(u(x)+1)=0
$$
Since there doesn't exist divisor of zero in $A[S]$, we get
($x\not\in\ker N$):
$$
x^{\star}=x\,\Leftrightarrow\,u(x)-1=0\,\,\,\,\mathrm{or}\,\,\,\,u(x)+1=0
$$
$$
x^{\star}=x\,\Leftrightarrow\,u(x)=1\,\,\,\,\mathrm{or}\,\,\,\,u(x)=-1
$$
\end{proof}

\begin{property}\label{p4.40}
Let $x$ be a non-zero element of $A[S]\setminus\ker N$. If there
exists an units $i$ of 
$U(A[S])$ such that $i^2+1=0$, then
we have ($x\not\in\ker N$): 
$$
x^{\star}=-x\,\Leftrightarrow\,u(x)\in\{i,-i\}
$$
\end{property}

\begin{proof}
Let $x$ be an element of $A[S]\setminus\ker N$. 
\\[0.1in]
Using the property (\ref{p.9}) of the definition (\ref{d4.25}) of the
unit part of an element $x$ of $A[S]\setminus\ker N$, we have
($x\not\in\ker N$):
$$
x^{\star}=-x\,\Leftrightarrow\,
u(x^{\star})N(x^{\star})=-u(x)N(x)
\,\Leftrightarrow\,
u(x^{\star})N(x)=-u(x)N(x)
$$
Since for $x\not\in\ker N$, $N(x)$ is regular (see the remark
(\ref{r4.10})), we have ($x\not\in\ker N$):
$$
x^{\star}=-x\,\Leftrightarrow\,u(x^{\star})=-u(x)\,\Leftrightarrow\,
u(x)^{-1}=-u(x)\,\Leftrightarrow\,u(x)^2=-1
$$
$$
x^{\star}=-x\,\Leftrightarrow\,u(x)^2+1=0
$$
Assuming that there exists an unit $i$ in $U(A[S])$ such that
$i^2+1=0$, we have ($x\not\in\ker N$): 
$$
(u(x)-i)(u(x)+i)=0
$$
Since there doesn't exist divisor of zero in $A[S]$, we get
($x\not\in\ker N$):
$$
u(x)-i=0\,\,\,\,\mathrm{or}\,\,\,\,u(x)+i=0
$$
$$
u(x)=i\,\,\,\,\mathrm{or}\,\,\,\,u(x)=-i
$$
It results that ($x\not\in\ker N$):
$$
x^{\star}=-x\,\Leftrightarrow\,u(x)\in\{i,-i\}
$$
\end{proof}

\begin{remark}\label{r4.41}
Assuming that there exists an unit $i$ in $U(A[S])$ such that
$i^2+1=0$, we have:
$$
N(i^2)=N(-1)=N(1)=1
$$
Since $N(i^2)=N(i)^2$, it follows that:
$$
N(i)^2=1\,\Leftrightarrow\,N(i)^2-1=0\,\Leftrightarrow\,
(N(i)-1)(N(i)+1)=0
$$
Since there doesn't exist divisor of zero in $A[S]$, we get:
$$
N(i)^2=1\,\Leftrightarrow\,N(i)-1=0\,\,\,\,\mathrm{or}\,\,\,\,
N(i)+1=0
$$
$$
N(i)^2=1\,\Leftrightarrow\,N(i)=1\,\,\,\,\mathrm{or}\,\,\,\,
N(i)=-1
$$
But, if $N(i)=-1$, then $N(N(i))=N(i)=N(-1)=1$. We reach to
a contradiction meaning that:
$$
N(i)=1
$$
It results that:
$$
u(i)=i
$$
From the property (\ref{p4.40}), we deduce that:
$$
i^{\star}=-i
$$
\end{remark}

\begin{property}\label{p4.42}
Let $x$ be an element of $A[S]\setminus\ker N$. Then in
$\mathbb{F}$, we have ($x\not\in A[S]\setminus\ker N$):
$$
xu(x^{\star})=x^{\star}u(x)
$$
\end{property}

\begin{proof}
From the property (\ref{p.9}) of the definition (\ref{d4.25}), using
the theorem (\ref{t4.38}), we have
in $\mathbb{F}$ ($x\not\in\ker N$): 
$$
u(x)x^{\star}=u(x)N(x^{\star})u(x^{\star})=u(x)N(x)u(x^{\star})=xu(x^{\star})
$$
\end{proof}

\begin{example}\label{e4.43}
Let $A=\mathbb{Z}$ be the ring of integers, let $S=\{i\}$ such
that $i^2+1=0$ and let $\mathbb{F}=\mathbb{C}$ equipped with the
usual modulus norm $||\,\,||_{\mathbb{C}}$. So, $A[S]$ is the subring
$\mathbb{Z}[i]$ 
of gaussian integers in $\mathbb{C}$. Let calculate the magnitude and
the unit part of the element $1+i$. Using the property (\ref{p4.42}), we have:
$$
(1+i)u(1-i)=(1-i)u(1+i)
$$
Since:
$$
1-i=-i(1+i)
$$
and since $-i=i^{\star}=i^{-1}\in U(A[S])$, we can notice that:
$$
\mathcal{D}(1-i)=\mathcal{D}(-i(1+i))=\mathcal{D}(1+i)
$$
From the property (\ref{p4.18}), there exists a unique element $u$ of
$U(A[S])$ such that:
$$
1-i=u(1+i)
$$
Since $1-i=-i(1+i)$, we have $u=-i$. From the equality
$(1+i)u(1-i)=(1-i)u(1+i)$, since the elements $1\pm i\neq 0$ are
regulars (see the remark (\ref{r4.10}), we have:
$$
u(1-i)=-iu(1+i)
$$
$$
u(1-i)u^{-1}(1+i)=-i
$$
$$
u(1-i)^2=-i=\left(\frac{1-i}{\sqrt{2}}\right)^2
$$
$$
u(1-i)^2-\left(\frac{1-i}{\sqrt{2}}\right)^2=0
$$
$$
\left(u(1-i)-\frac{1-i}{\sqrt{2}}\right)
\left(u(1-i)+\frac{1-i}{\sqrt{2}}\right)=0
$$
Since $A[S]$ is an entire ring, we get:
$$
u(1-i)-\frac{1-i}{\sqrt{2}}=0\,\,\,\,\mathrm{or}\,\,\,\,
u(1-i)+\frac{1-i}{\sqrt{2}}=0
$$
$$
u(1-i)=\frac{1-i}{\sqrt{2}}\,\,\,\,\mathrm{or}\,\,\,\,
u(1-i)=-\frac{1-i}{\sqrt{2}}
$$
If absurdly $u(1-i)=-\frac{1-i}{\sqrt{2}}$, then we would have: 
$$
1-i=-\frac{1-i}{\sqrt{2}}N(1-i)
$$
$$
(1-i)\left\{1+\frac{N(1-i)}{\sqrt{2}}\right\}=0
$$
Since $1-i\neq 0$ is regular, it would result that:
$$
1+\frac{N(1-i)}{\sqrt{2}}=0
$$
Since $\sqrt{2}\in\mathbb{R}$ is invertible in $\mathbb{R}$, it would
imply that: 
$$
N(1-i)=-\sqrt{2}
$$
and so:
$$
||N(1-i)||_{\mathbb{C}}=\left|\left|-\sqrt{2}\right|\right|_{\mathbb{C}}=\sqrt{2}
$$
But it would contradict the fact that (see the property (\ref{p.7}) of
the definition (\ref{d4.3})):
$$
||N(1-i)||_{\mathbb{C}}=N(1-i)
$$
It means that:
$$
u(1-i)=\frac{1-i}{\sqrt{2}}
$$
which implies that:
$$
N(1-i)=\sqrt{2}
$$
Therefore:
$$
u(1+i)=\frac{1+i}{\sqrt{2}}
$$
and:
$$
N(1+i)=\sqrt{2}
$$
\end{example}

\begin{remark}\label{r4.44}
Using the properties (\ref{p.11}) and (\ref{p.12}), we can notice that:
$$
(x+x^{\star})^{\star}=x^{\star}+x
$$
From the property (\ref{p4.39}), it means that:
$$
u(x+x^{\star})\in\{-1;1\}
$$
Using the property (\ref{p.9}) of the definition (\ref{d4.25}) and
using the definition (\ref{d4.3}), it gives:
$$
x+x^{\star}=\pm N(x+x^{\star})\in\mathbb{F}\cap\mathbb{R}
$$
\end{remark}

\noindent
Thus, we can set the following definition.

\begin{definition}\label{d4.45}
The real function\index{real function $\mathrm{Re}$} $\mathrm{Re}$ is
the map defined on $A[S]$ 
which associates the element $\mathrm{Re}(x)=\frac{x+x^{\star}}{2}$ of
the set $\mathbb{F}\cap\mathbb{R}$, to $x\in A[S]$:
\begin{eqnarray}
\mathrm{Re}:A[S]&\longrightarrow& \mathbb{F}\cap\mathbb{R}
\nonumber\\
x&\mapsto& \mathrm{Re}(x)=\frac{x+x^{\star}}{2}
\nonumber
\end{eqnarray}
The number $\mathrm{Re}(x)$ is called the real part of the element $x$
of $A[S]$.
\end{definition}

\begin{property}\label{p4.46}
Let $x$ be an element of $A[S]$. Then, we have:
$$
\mathrm{Re}(x^{\star})=\mathrm{Re}(x)
$$
\end{property}

\begin{proof}
The property (\ref{p4.46}) follows from the definition (\ref{d4.45})
of the real function $\mathrm{Re}$. 
\end{proof}

\begin{property}\label{p4.47}
Let $x,y$ be two elements of $A[S]$. Then we have:
$$
\mathrm{Re}(x+y)=\mathrm{Re}(x)+\mathrm{Re}(y)
$$
\end{property}

\begin{proof}
The property (\ref{p4.47}) follows from the definition (\ref{d4.45})
of the real function $\mathrm{Re}$. 
\end{proof}

\begin{property}\label{p4.48}
Let $a$ be an element of $A[S]\setminus U(A[S])$. Then for all
$x\in A[S]$, we have ($a\not\in U(A[S])$):
$$
\mathrm{Re}(ax)=a\mathrm{Re}(x)
$$
\end{property}

\begin{proof}
Let $a$ be an element of $A[S]\setminus U(A[S])$. From the definition
(\ref{d4.36}), we have $a^{\star}=a$. Using the
definition (\ref{d4.45}) of the real function, we have ($x\in A[S]$
and $a\not\in U(A[S])$):
$$
\mathrm{Re}(ax)=\frac{ax+x^{\star}a^{\star}}{2}=\frac{ax+x^{\star}a}{2}
$$
$$
\mathrm{Re}(ax)=\frac{ax+ax^{\star}}{2}=a\frac{x+x^{\star}}{2}
=a\mathrm{Re}(x)
$$
\end{proof}

\begin{remark}\label{r4.49}
Using the properties (\ref{p.11}) and (\ref{p.12}) with the remark
(\ref{r4.41}), we can notice that:
$$
(-i(x-x^{\star}))^{\star}=(x^{\star}-x)(-i)^{\star}=-i(x-x^{\star})
$$
From the property (\ref{p4.39}), it means that:
$$
u(x+x^{\star})\in\{-1;1\}
$$
Using the property (\ref{p.9}) of the definition (\ref{d4.25}) and
using the definition (\ref{d4.3}), it gives:
$$
-i(x-x^{\star})=\pm N(x+x^{\star})\in\mathbb{F}\cap\mathbb{R}
$$
\end{remark}

\noindent
Thus, we can set the following definition.

\begin{definition}\label{d4.50}
The imaginary function\index{imaginary function $\mathrm{Im}$}
$\mathrm{Im}$ is the map defined on $A[S]$ 
which associates the element $\mathrm{Im}(x)=-i\frac{x-x^{\star}}{2}$ of
the set $\mathbb{F}\cap\mathbb{R}$, to $x\in A[S]$:
\begin{eqnarray}
\mathrm{Im}:A[S]&\longrightarrow& \mathbb{F}\cap\mathbb{R}
\nonumber\\
x&\mapsto& \mathrm{Im}(x)=-i\frac{x-x^{\star}}{2}
\nonumber
\end{eqnarray}
The number $\mathrm{Im}(x)$ is called the imaginary part of the element $x$
of $A[S]$.
\end{definition}

\begin{property}\label{p4.51}
Let $x$ be an element of $A[S]$. Then, we have:
$$
\mathrm{Im}(x^{\star})=-\mathrm{Im}(x)
$$
\end{property}

\begin{proof}
The property (\ref{p4.51}) follows from the definition (\ref{d4.50})
of the real function $\mathrm{Im}$. 
\end{proof}

\begin{property}\label{p4.52}
Let $x,y$ be two elements of $A[S]$. Then we have:
$$
\mathrm{Im}(x+y)=\mathrm{Im}(x)+\mathrm{Im}(y)
$$
\end{property}

\begin{proof}
The property (\ref{p4.52}) follows from the definition (\ref{d4.50})
of the imaginary function $\mathrm{Im}$. 
\end{proof}

\begin{property}\label{p4.53}
Let $a$ be an element of $A[S]\setminus U(A[S])$. Then for all
$x\in A[S]$, we have ($a\not\in U(A[S])$):
$$
\mathrm{Im}(ax)=a\mathrm{Im}(x)
$$
\end{property}

\begin{proof}
Let $a$ be an element of $A[S]\setminus U(A[S])$. So, we have 
$a^{\star}=a$. Using the
definition (\ref{d4.50}) of the imaginary function, we have ($x\in
A[S]$ and $a\not\in U(A[S])$):
$$
\mathrm{Im}(ax)=-i\frac{ax-x^{\star}a^{\star}}{2}=-i\frac{ax-x^{\star}a}{2}
$$
$$
\mathrm{Im}(ax)=-i\frac{ax-ax^{\star}}{2}=a\left(-i\frac{x-x^{\star}}{2}\right)
=a\mathrm{Im}(x)
$$
\end{proof}

\begin{theorem}\label{t4.54}
For all $x\in A[S]$, we have:
$$
x=\mathrm{Re}(x)+i\mathrm{Im}(x)
$$
\end{theorem}

\begin{proof}
Using the definitions (\ref{d4.45}) and (\ref{d4.50}), for all $x\in
A[S]$, we have: 
$$
\mathrm{Re}(x)+i\mathrm{Im}(x)=\frac{x+x^{\star}}{2}+i(-i)\frac{x-x^{\star}}{2}
$$
Since $i^2+1=0$, it gives:
$$
\mathrm{Re}(x)+i\mathrm{Im}(x)=\frac{x+x^{\star}}{2}+\frac{x-x^{\star}}{2}
=\frac{x+x^{\star}+x-x^{\star}}{2}=\frac{2x}{2}=x
$$
So, we deduce that:
$$
\mathrm{Re}(x)+i\mathrm{Im}(x)=x
$$
\end{proof}

\begin{corollary}\label{c4.55}
Let $x$ be an element of $A[S]$. Then, we have:
$$
x^{\star}=\mathrm{Re}(x)-i\mathrm{Im}(x)
$$
\end{corollary}

\begin{proof}
From the theorem (\ref{t4.54}), we have:
$$
x^{\star}=\mathrm{Re}(x^{\star})+i\mathrm{Im}(x^{\star})
$$
Using the properties, it results that:
$$
x^{\star}=\mathrm{Re}(x)-i\mathrm{Im}(x)
$$
\end{proof}

\begin{remark}\label{r4.56}
Since $0^{\star}=0$, using also the theorem (\ref{t4.54}) and the
corollary (\ref{c4.55}), we have: 
$$
x=0\,\Rightarrow\,x^{\star}=0^{\star}=0
$$
and:
$$
\mathrm{Re}(x)+i\mathrm{Im}(x)=0
$$
$$
\mathrm{Re}(x)-i\mathrm{Im}(x)=0
$$
Taking the sum and the difference side by side of these two equations,
since $2\neq 0$ is regular, we deduce that:
$$
x=0\,\Leftrightarrow\,\left\{\begin{array}{c}
\mathrm{Re}(x)=0
\\[0.1in]
\mathrm{and}
\\[0.1in]
\mathrm{Im}(x)=0
\end{array}
\right.
$$
\end{remark}

\begin{definition}\label{d4.57}
Let $k\in\mathbb{N}^{\star}$ and 
$\{f_1,\ldots,f_k\}$ be a family of a subfield $\mathbb{F}$ of
$\mathbb{C}$. We said that the family
$\{f_1,\ldots,f_k\}$ is free over a subring $A$ of
$\mathbb{C}$, if for $a_1,\ldots,a_k\in A$:
$$
{\displaystyle\sum^k_{i=1}}a_if_i=0\,\Rightarrow\,a_1=\ldots=a_k=0
$$
In other words, when a family $\{f_1,\ldots,f_k\}$ of a subfield $\mathbb{F}$ of
$\mathbb{C}$ is free over a subring $A$ of
$\mathbb{C}$, it means that the elements $f_1,\ldots,f_k$
are linearly independent over $A$.
\end{definition}

\begin{theorem}\label{t4.58}
Let $e=\{e_0,e_1,\ldots,e_{n-1}\}$ with $n\in\mathbb{N}^{\star}$ 
be a maximal free family of elements of $G$ over $\mathbb{C}$
with:
$$
e_0=e_n=1
$$
such that ($i=1,\ldots,n-1$):
$$
\mathrm{Im}(e_i)\neq 0
$$
If for all $a\in A$, $a^{\star}=a$, then the image of the free 
family $e\setminus\{e_0\}=\{e_1,\ldots,e_{n-1}\}$ with $n\in\mathbb{N}^{\star}$,
of elements of $G$ under $\mathrm{Im}$, is
a free family of elements of $\mathbb{F}\cap\mathbb{R}$ over $A$.
\end{theorem}

\begin{proof}
Let $e=\{e_0,e_1,\ldots,e_{n-1}\}$ with $n\in\mathbb{N}^{\star}$ be a
a maximal free family of elements of $G$ over $\mathbb{C}$ with:
$$
e_0=e_n=1
$$
such that ($i=1,\ldots,n-1$):
$$
\mathrm{Im}(e_i)\neq 0
$$ 
If for elements $a_1,\ldots,a_{n-1}$ of $A$ with $n\in\mathbb{N}^{\star}$:
$$
{\displaystyle\sum^{n-1}_{i=1}}a_i\mathrm{Im}(e_i)=0
$$
then since $a^{\star}_i=a_i$ for all $i\in\{1,\ldots,n-1\}$, using the
properties (\ref{p4.52}) and (\ref{p4.53}), we have ($n\in\mathbb{N}^{\star}$):
$$
\mathrm{Im}\left({\displaystyle\sum^{n-1}_{i=1}}a_ie_i\right)=0
$$
which implies that (see also the properties (\ref{p4.47}) and (\ref{p4.48}), 
$n\in\mathbb{N}^{\star}$):
$$
{\displaystyle\sum^{n-1}_{i=1}}a_ie_i=
\mathrm{Re}\left({\displaystyle\sum^{n-1}_{i=1}}a_ie_i\right)
={\displaystyle\sum^{n-1}_{i=1}}a_i\mathrm{Re}(e_i)
$$
Since $\mathrm{Re}(e_i)\in\mathbb{F}\cap\mathbb{R}$, then there exists
an element $c_0$ of 
$\mathbb{F}\cap\mathbb{R}$ such that ($n\in\mathbb{N}^{\star}$):
$$
c_0e_0={\displaystyle\sum^{n-1}_{i=1}}a_i\mathrm{Re}(e_i)
$$
It gives ($n\in\mathbb{N}^{\star}$):
$$
{\displaystyle\sum^{n-1}_{i=1}}a_ie_i=c_0e_0
$$
Since  $e=\{e_0,e_1,\ldots,e_{n-1}\}$ is a free family over
$\mathbb{C}$, it implies 
that ($n\in\mathbb{N}^{\star}$): 
$$
c_0=a_1=\ldots=a_{n-1}=0
$$
So, we get ($n\in\mathbb{N}^{\star}$ and $a_1,\ldots,a_{n-1}\in A$):
$$
{\displaystyle\sum^{n-1}_{i=1}}a_i\mathrm{Im}(e_i)=0\,\Rightarrow\,
a_1=\ldots=a_{n-1}=0
$$
It means that $\{\mathrm{Im}(e_1),\ldots,\mathrm{Im}(e_{n-1})\}$ with
$n\in\mathbb{N}^{\star}$ is a free
family of elements of $\mathbb{F}\cap\mathbb{R}$ over $A$.
\end{proof}

\begin{theorem}\label{t4.59}
Let $e=\{e_0,e_1,\ldots,e_{n-1}\}$ with $n\in\mathbb{N}^{\star}$ 
be a maximal free family of elements of $G$ over $\mathbb{C}$ with:
$$
e_0=e_n=1
$$
such that ($i=1,\ldots,n-1$):
$$
\mathrm{Re}(e_i)\neq 0
$$
If for all $a\in A$, $a^{\star}=a$, then the image of the free
family $e\setminus\{e_0\}=\{e_1,\ldots,e_{n-1}\}$ with $n\in\mathbb{N}^{\star}$,
of elements of $G$ under $\mathrm{Re}$, is
a free family of elements of $\mathbb{F}\cap\mathbb{R}$ over $A$.
\end{theorem}

\begin{proof}
Let $e=\{e_0,e_1,\ldots,e_{n-1}\}$ with $n\in\mathbb{N}^{\star}$ be a
a maximal free family of elements of $G$ with:
$$
e_0=e_n=1
$$
such that ($i=1,\ldots,n-1$):
$$
\mathrm{Re}(e_i)\neq 0
$$ 
If for elements $a_1,\ldots,a_{n-1}$ of $A$ with $n\in\mathbb{N}^{\star}$:
$$
{\displaystyle\sum^{n-1}_{i=1}}a_i\mathrm{Re}(e_i)=0
$$
then since $a^{\star}_i=a_i$ for all $i\in\{1,\ldots,n-1\}$, using the
properties (\ref{p4.47}) and (\ref{p4.48}), we have ($n\in\mathbb{N}^{\star}$):
$$
\mathrm{Re}\left({\displaystyle\sum^{n-1}_{i=1}}a_ie_i\right)=0
$$
which implies that (see also the properties (\ref{p4.52}) and (\ref{p4.53}), 
$n\in\mathbb{N}^{\star}$):
$$
{\displaystyle\sum^{n-1}_{i=1}}a_ie_i=
i\mathrm{Im}\left({\displaystyle\sum^{n-1}_{i=1}}a_ie_i\right)
=i{\displaystyle\sum^{n-1}_{i=1}}a_i\mathrm{Im}(e_i)
$$
Since $\mathrm{Im}(e_i)\in\mathbb{F}\cap\mathbb{R}$, then there exists
an element $c_0$ of 
$\mathbb{F}\cap\mathbb{R}$ such that ($n\in\mathbb{N}^{\star}$):
$$
c_0e_0={\displaystyle\sum^{n-1}_{i=1}}a_i\mathrm{Im}(e_i)
$$
It gives ($n\in\mathbb{N}^{\star}$):
$$
{\displaystyle\sum^{n-1}_{i=1}}a_ie_i=ic_0e_0
$$
Since  $e=\{e_0,e_1,\ldots,e_{n-1}\}$ is a free family over
$\mathbb{C}$, it implies that ($n\in\mathbb{N}^{\star}$): 
$$
c_0=a_1=\ldots=a_{n-1}=0
$$
So, we get ($n\in\mathbb{N}^{\star}$ and $a_1,\ldots,a_{n-1}\in A$):
$$
{\displaystyle\sum^{n-1}_{i=1}}a_i\mathrm{Im}(e_i)=0\,\Rightarrow\,
a_1=\ldots=a_{n-1}=0
$$
It means that $\{\mathrm{Re}(e_1),\ldots,\mathrm{Re}(e_{n-1})\}$ with
$n\in\mathbb{N}^{\star}$ is a free
family of elements of $\mathbb{F}\cap\mathbb{R}$ over $A$.
\end{proof}

\begin{property}\label{p4.60}
Let $x$ be an element of $A[S]$. Then we have:
$$
xx^{\star}=N(x)^2
$$
\end{property}

\begin{proof}
Let $x$ be an element of $A[S]$. If $x\in\ker N$, then since $0^{\star}=0$
(see the remark (\ref{r4.37})), using the property (\ref{p.2}) of the
definition (\ref{d4.3}), we
have ($x\in\ker N$): 
$$
00^{\star}=00=0=N(x)
$$
In the following, we assume that $x\not\in\ker N$.
From the property (\ref{p.9}) of the
definition (\ref{d4.25}), using the theorem (\ref{t4.38}), we have
($x\not\in\ker N$): 
$$
xx^{\star}=u(x)N(x)u(x^{\star})N(x^{\star})=u(x)N(x)u(x)^{-1}N(x)
$$
$$
xx^{\star}=u(x)u(x)^{-1}N(x)N(x)=N(x)^2
$$
\end{proof}

\begin{remark}\label{r4.61}
Let $x,y$ be two elements of $A[S]$. Then we have:
$$
N(x+y)^2=(x+y)(x+y)^{\star}=(x+y)(x^{\star}+y^{\star})
$$
$$
N(x+y)^2=xx^{\star}+xy^{\star}+yx^{\star}+yy^{\star}
$$
$$
N(x+y)^2=N(x)^2+2\mathrm{Re}(xy^{\star})+N(y)^2
$$
So, $N$ satisfies a triangular inequality if, and only if 
$\mathrm{Re}(xy^{\star})\leq N(xy)$. In this case, $N$ behaves as a
norm on $A[S]$. 
\end{remark}

\begin{definition}\label{d4.62}
The radius function is the function defined on $A[S]$ which associates
the unique element $\sqrt{xx^{\star}}$ of $\mathbb{R}_+$ to $x\in A[S]$:
\begin{eqnarray}
r:A[S]&\longrightarrow&\mathbb{R}_+
\nonumber\\
x&\mapsto& r(x)=\sqrt{xx^{\star}}
\nonumber
\end{eqnarray}
with property:
\[
||r(x)||_{\mathbb{F}}=r(x)\label{p.14}\tag{P.14}
\]
\end{definition}

\begin{corollary}\label{c4.63}
Let $x$ be an element of $A[S]$. Then, we have:
$$
N(x)=r(x)
$$
\end{corollary}

\begin{proof}
Let $x$ be an element of $A[S]$. Using the property
(\ref{p4.60}), we have ($x\in A[S]$):
$$
N(x)^2=xx^{\star}=r^2(x)
$$
So, either $N(x)=-r(x)$ or
$N(x)=r(x)$. If absurdly, $N(x)=-r(x)$,
from the property
(\ref{p.7}) of the definition (\ref{d4.3}), then we have:
$$
||N(x)||_{\mathbb{F}}=N(x)=-r(x)
$$ 
Using the property (\ref{p.14}) of the definition (\ref{d4.62}), it gives:
$$
||N(x)||_{\mathbb{F}}=||-r(x)||_{\mathbb{F}}=||r(x)||_{\mathbb{F}}=r(x)
$$
It results that $||N(x)||_{\mathbb{F}}=r(x)$ which contradicts the
assumption. Therefore, we have:
$$
||N(x)||_{\mathbb{F}}=N(x)=r(x)
$$ 
\end{proof}

\noindent
Recall that
any non-zero element of $A[S]$ is invertible in $\mathbb{F}$. Thus, for
any non-zero element $x$ of $A[S]$, the fraction $\frac{1}{x}$
defined on $\mathbb{F}$ means also the inverse $x^{-1}$ in
$\mathbb{F}$.

\begin{property}\label{p4.64}
Let $x$ be an element of $A[S]\setminus\ker N$. Then in 
$\mathbb{F}$, we have ($x\not\in\ker N$):
$$
\frac{1}{x}=\frac{1}{N(x)^2}x^{\star}
$$
\end{property}

\begin{proof}
Let $x$ be an element of $A[S]\setminus\ker N$. We have obviously:
$$
xx^{\star}=xx^{\star}
$$
Since $x\neq 0$ is assumed to be invertible in $\mathbb{F}$, then in 
$\mathbb{F}$, we have ($x\not\in\ker N$): 
$$
\frac{1}{x}xx^{\star}=x^{\star}
$$
Since $xx^{\star}\neq 0$ is invertible in $\mathbb{F}$, then from
the property (\ref{p4.60}), in $\mathbb{F}$, we have ($x\not\in\ker N$):
$$
\frac{1}{x}=x^{\star}\frac{1}{xx^{\star}}
$$
$$
\frac{1}{x}=\frac{1}{xx^{\star}}x^{\star}=\frac{1}{N(x)^2}x^{\star}
$$
\end{proof}

\begin{property}\label{p4.65}
Let $x$ be an element of $A[S]\setminus\ker N$. Then in 
$\mathbb{F}$, we have ($x\not\in\ker N$):
$$
u(x)=\frac{1}{N(x)}x
$$
\end{property}

\begin{proof}
Let $x$ be an element of $A[S]\setminus\ker N$. Using the definition
(\ref{d4.25}), we have ($x\not\in\ker N$):
$$
x=N(x)u(x)
$$
Then in 
$\mathbb{F}$, we have ($x\not\in\ker N$):
$$
\frac{1}{N(x)}x=u(x)
$$
\end{proof}

\begin{property}\label{p4.66}
If $||\,\,||_{\mathbb{F}}$ is an extension of $r$ to $\mathbb{F}$,
then for all $x\not\in\ker N$ of $A[S]$, in $\mathbb{F}$, we have
($x\not\in\ker N$):
$$
||u(x)||_{\mathbb{F}}=1
$$
\end{property}

\begin{proof}
Let $x$ be an element of $A[S]\setminus\ker N$. Then in $\mathbb{F}$,
we have ($x\not\in\ker N$):
$$
||u(x)||_{\mathbb{F}}=\left|\left|\frac{1}{N(x)}x\right|\right|_{\mathbb{F}}
=\left|\left|\frac{1}{N(x)}\right|\right|_{\mathbb{F}}||x||_{\mathbb{F}}
$$
$$
||u(x)||_{\mathbb{F}}=\frac{1}{||N(x)||_{\mathbb{F}}}||x||_{\mathbb{F}}
=\frac{1}{r(x)}r(x)=1
$$
\end{proof}

\begin{property}\label{p4.67}
Let $x$ be an element of $A[S]\setminus\ker N$. If $||\,\,||_{\mathbb{F}}$ is
an extension of $r$ to $\mathbb{F}$ and if $\,\overline{\phantom{x}}$ is an
extension of $\star$ operation to $\mathbb{F}$ such that in
$\mathbb{F}$ ($x\not\in\ker N$):
$$
\overline{\frac{1}{x}}=\frac{1}{||x||^2_{\mathbb{F}}}x
$$
then in $\mathbb{F}$ ($x\not\in\ker N$):
$$
\overline{\frac{1}{x}}=\frac{1}{\overline{x}}
$$
\end{property}

\begin{proof}
Let $x$ be an element of $A[S]\setminus\ker N$. From the property
(\ref{p4.64}), we have ($x\not\in\ker N$):
$$
\frac{1}{x^{\star}}=\frac{1}{N(x^{\star})^2}(x^{\star})^{\star}
$$
Using the property (\ref{p.11}) of the definition
(\ref{d4.36}) and the theorem (\ref{t4.38}), it comes that ($x\not\in\ker N$):
$$
\frac{1}{x^{\star}}=\frac{1}{N(x)^2}x
$$
Since $||\,\,||_{\mathbb{F}}$ is an extension of $r$ to $\mathbb{F}$,
using the corollary (\ref{c4.63}),
it gives ($x\not\in\ker N$):
$$
\frac{1}{x^{\star}}=\frac{1}{||x||^2_{\mathbb{F}}}x=\overline{\frac{1}{x}}
$$
Or, since $\overline{\phantom{x}}$ is an
extension of $\star$ operation to $\mathbb{F}$, we obtain
($x\not\in\ker N$):
$$
\frac{1}{\overline{x}}=\overline{\frac{1}{x}}
$$
\end{proof}

\section{{The fundamental theorem of arithmetic in an entire 
ring}}\label{s5}

In this section, we shall consider an entire subring $A$ of a subfield
$\mathbb{F}$ of $\mathbb{C}$ such that:
$$
\forall\,\,a\in A\setminus\{0\},\,\,a\not\in\ker N
$$
$$
\forall\,\,a\in A,\,\,N(a)\in A
$$

\begin{definition}\label{d5.1}
Let $a,b$ be two elements of $A$.
\\[0.1in]
Let $\mathcal{D}(a,b)$ be the set of
common divisors of $a,b$. If there exists an element $g$ of
$\mathcal{D}(a,b)$ such that whatever $d\in\mathcal{D}(a,b)$, $d|g$
and if $N(g)$ belongs to $A\setminus (U(A)\setminus\{1\})$,
then $N(g)$ is defined as the greatest common divisor\index{greatest
common divisor} of $a,b$. The
element $N(g)$ of $\mathcal{D}(a,b)$ is denoted $\gcd(a,b)$. Of course,
if $\gcd(a,b)$ exists, then $\gcd(a,b)=\gcd(b,a)$.
\\[0.1in]
If $\mathcal{D}(a,b)=U(A)$, by convention, we set $\gcd(a,b)=1$. In
such a case, the two elements $a,b$ are said to be relatively
primes. In particular, $\gcd(a,b)=1$ for all $a\in U(A)$ and for all
$b\in A$. We have also $\gcd(1,a)=1$ for all $a\in A$. Moreover, for
any $v,v'\in U(A)$, $\mathcal{D}(v,v')=U(A)$ and $\gcd(v,v')=1$.
\\[0.1in]
A non-zero element $p$ of $A$ is said to be
irreducible\index{irreducible element} if, and only if, 
$p\not\in U(A)\cup\{0\}$ and we have 
$\mathcal{D}(p)=U(A)\cup pU(A)$. Moreover, a non-zero irreducible
element $p$ of $A$ is said to be a prime\index{prime} if, and only if,
$N(p)=p$. By convention,
the elements of $U(A)$ are not irreducible. So, any irreducible
element of $A$ is not
invertible. Moreover, when $p$ of $A$ is prime, we have 
$\gcd(a,p)=1$ if $p\!\not|\,a$ and $\gcd(a,p)=p$ if $p|a$.
\\[0.1in]
Let $\mathcal{M}(a,b)$ be the set of
common multiples of $a,b$. If there exists an element $\ell$ of
$\mathcal{M}(a,b)$ such that whatever $m\in\mathcal{M}(a,b)$, $\ell|m$
and if $N(\ell)$ belongs to $A\setminus (U(A)\setminus\{1\})$,
then $N(\ell)$ is defined as the least common multiple\index{least
common multiple} of $a,b$. The
element $N(\ell)$ of $\mathcal{M}(a,b)$ is denoted $\lcm(a,b)$. Of course,
if $\lcm(a,b)$ exists, then $\lcm(a,b)=\lcm(b,a)$.
\end{definition}

\begin{remark}\label{r5.2}
$$\mathcal{D}(a,b)=\mathcal{D}(a)\cap\mathcal{D}(b)$$
$$U(A)\subseteq\mathcal{D}(a,b)$$
$$\mathcal{M}(a,b)=\mathcal{M}(a)\cap\mathcal{M}(b)=aA\cap bA$$
\end{remark}

\begin{property}\label{p5.3}
If $A$ contains at least a prime element and if $U(A)$ has finite order,
then any non-zero element of $A$ which doesn't belong to
$U(A)\cup\{0\}$ such that $|\mathcal{D}(x)|$ is finite, has a prime divisor.
\end{property}

\begin{proof}
We assume that $A$ contains at least a prime element.
\\[0.1in]
Let $x$ be a non-zero element of $A$ which does not belong to
$U(A)\cup\{0\}$ such that $|\mathcal{D}(x)|$ is finite. Since
$x\in\mathcal{D}(x)$, $\mathcal{D}(x)$ is 
non-empty. If $x$ is prime, we find a prime divisor of $x$, namely $x$
itself. Using the remark (\ref{r3.3}), there exists a
non-zero element $d_1$ 
of $\mathcal{D}(x)$ which does not belong to $U(A)$ such that
$\mathcal{D}(d_1)\subseteq\mathcal{D}(x)$ and: 
$$
x=a_1d_1
$$
with $a_1\in A\setminus\{0\}$. Notice that $\mathcal{D}(d_1)$ is
non-empty since $d_1\in\mathcal{D}(d_1)$ and $\gcd(a_1,d_1)$ is not
necessarily equal to $1$. Notice also that if we cannot find $d_1$ in
$\mathcal{D}(x)$ such that $a_1\not\in U(A)$, then
$\mathcal{D}(x)=U(A)\cup xU(A)$ and so $N(x)$ is prime. If $d_1$ is
not a prime element of $A$ and if $a_1\not\in U(A)$, since $d_1$ is a
non-zero element of $A$ which does not belong 
to $U(A)$, there exists a non-zero element $d_2$ of
$\mathcal{D}(d_1)$ which does not belong to $U(A)$ such that
$\mathcal{D}(d_2)\subseteq\mathcal{D}(d_1)$ with
$\mathcal{D}(d_2)\neq\emptyset$ and: 
$$
d_1=a_2d_2
$$
So:
$$
x=a_1a_2d_2
$$
with $a_1\in A\setminus (U(A)\cup\{0\})$ and
$a_2\in A\setminus\{0\}$. Notice that $\gcd(a_2,d_2)$ as well as 
$\gcd(a_1a_2,d_2)$ is not necessarily equal to $1$. Notice also that
if we cannot find $d_2$ in 
$\mathcal{D}(d_1)$ such that $a_2\not\in U(A)$, then
$\mathcal{D}(d_1)=U(A)\cup d_1U(A)$ and so $N(d_1)$ is prime.
If $d_2$ is not a prime element of $A$ and $a_2\not\in U(A)$,
we follow the same steps than above. Thus, we get a sequence
$(\mathcal{D}(d_i))$ of nested non-empty subsets of $A$ such
that ($x\not\in U(A)\cup\{0\}$):
$$
\mathcal{D}(d_0)=\mathcal{D}(x)
$$
$$
\mathcal{D}(d_{i+1})\subseteq\mathcal{D}(d_i)
\,\,\,\,\mathrm{with}\,\,d_i,d_{i+1}\in\mathcal{D}(x)\setminus U(A)
$$
$$
\mathcal{D}(d_i)\neq\emptyset\,\,\,\,\mathrm{with}\,\,
d_i\in\mathcal{D}(x)\setminus U(A)
$$
and ($i\geq 2$, 
$a_1,a_2,\ldots,a_{i-1}\in A\setminus (U(A)\cup\{0\})$ and
$a_i\in A\setminus\{0\}$):
$$
x=a_1\ldots a_{i-1}a_id_i
$$
Since $\mathcal{D}(x)$ has a finite order by assumption, the sequence 
$(\mathcal{D}(d_i))$ is finite. It follows that there exists
$n\in\mathbb{N}$ such that
$\mathcal{D}(d_{n+1})=\mathcal{D}(d_n)$. So, using the property
(\ref{p4.18}), $d_{n+1}|d_n$ and
$d_n|d_{n+1}$ meaning that $d_n$ and $d_{n+1}$ differ from a
multiplicative unit namely ($a_{n+1}=v\in U(A)$):
$$
d_n=ud_{n+1}
$$
It comes that ($a_1,a_2,\ldots,a_n\in A\setminus (U(A)\cup\{0\})$):
$$
x=a_1\ldots a_nd_n
$$
and ($d_n\not\in U(A)\cup\{0\}$ exists so $\mathcal{D}(d_n)\neq\emptyset$):
$$
\mathcal{D}(d_n)\subseteq\ldots\subseteq\mathcal{D}(d_0)
$$
The natural number $n$ is equal to the greatest integer for which
$d_n|x$:
$$
n=\max\{i\in\mathbb{N}\,:\,x=a_1\ldots
a_id_i\,\,\,\,\mathrm{with}\,\,\,\,a_1,\ldots,a_i\in 
A\setminus\{0\}\,\,\,\,\mathrm{and}\,\,\,\,d_i\in A\setminus (U(A)\cup\{0\})\}
$$
So, only $d_n$ and $d_{n+1}$ with their symmetric opposites $-d_n$ and $-d_{n+1}$ 
in $\mathcal{D}(d_n)$ do not belong to $U(A)$. Otherwise, there
exists a non-zero element $b$ in $\mathcal{D}(d_n)$ such that $b\neq v$ and
$b\neq vd_n$ with $v\in U(A)$,
which divides $d_n$, $d_{n+1}$ and so $x$ by transitivity of the relation
of divisibility defined on $A$. But, then $d_n=bc$ with $c\in
A\setminus (U(A)\cup\{0\})$ and $n$ would not be the greatest integer
such that $d_n|x$. We reach to a contradiction meaning that $b$
doesn't exist. It results that $\mathcal{D}(d_n)=U(A)\cup d_nU(A)$
with $d_n\not\in U(A)\cup\{0\}$. So, $N(d_n)$ is a prime element of
$A$ which divides $x$. It proved that $x$ has at least a prime
divisor.

\end{proof}

\begin{corollary}\label{c5.4}
If $A$ contains at least a prime element and if $U(A)$ has finite order,
then for any non-zero element $x$ of $A$ which
doesn't belong to $U(A)\cup\{0\}$ such that $|\mathcal{D}(x)|$ is finite, there 
exists a prime element $p$ of $A$ and a non-zero natural number $n$
such that:
$$
x=ap^n\,\,\,\,\mathrm{with}\,\,\,\,a\in A\setminus\{0\}
\,\,\,\,\mathrm{such\,\,that}\,\,\,\,\gcd(a,p)=1
$$ 
\end{corollary}

\begin{proof}
We assume that $A$ contains at least a prime element.
\\[0.1in]
Let $x\in A\setminus (U(A)\cup\{0\})$. From the property (\ref{p5.3}),
there exist a prime element $p$ and an element $b$ of
$A\setminus\{0\}$ such that $x=b_1p$. If $\gcd(b_1,p)=1$, then the
property is verified with $n=1$. If $p|b_1$, then there exists a non
zero element $b_2$ such that $b_1=b_2p$ and so $x=b_2p^2$. If
$\gcd(b_2,p)=1$, then the property is verified with $n=2$. If $p|b_2$,
we follow the steps above. Thus, we get a sequence $(b_i)$ of non-zero
elements of $A$ such that:
$$
b_{i+1}=b_ip
$$
and:
$$
x=b_ip^i
$$
with $i\in\mathbb{N}^{\star}$ and $b_i\in A\setminus (U(A)\cup\{0\})$.
Since $|\mathcal{D}(x)|$ is finite, there exists a non-zero natural
number $n$ such that $x=b_np^n$ and $p\!\not|\,b_n$. So, $b_n$ and $p$
are relatively primes which implies that $\gcd(b_n,p)=1$.
Setting $a=b_n$, we obtain $x=ap^n$ with $\gcd(a,p)=1$ and $a\in
A\setminus (U(A)\cup\{0\})$.

\end{proof}

\begin{theorem}[The fundamental theorem of
arithmetic in $A$]\label{t5.5}
Let $k\in\mathbb{N}^{\star}$.
\\[0.1in]
If $A$ contains at least a prime element and if $U(A)$ has finite order, 
any non-zero element $x$ of $A$ which does not belong to
$U(A)\cup\{0\}$ such that $|\mathcal{D}(x)|$ is finite, 
has a decomposition into prime factors up to a
multiplicative unit $v\in U(A)$ as:
$$
x=vp^{n_1}_1\ldots p^{n_k}_k
$$
where $p_1,\ldots,p_k$ which are primes such that $p_i\neq p_j$ for
$i\neq j$ with $i,j\in\llbracket 1,k\rrbracket$
and $n_1,\ldots,n_k\in\mathbb{N}^{\star}$.
\\[0.1in]
This decomposition is unique up to the order of factors.
\end{theorem}

\begin{proof}
Let $k\in\mathbb{N}^{\star}$.
\\[0.1in]
We assume that $A$ contains at least a prime element.
\\[0.1in]
Let $x$ be an element of $A\setminus (U(A)\cup\{0\})$. 
From the corollary (\ref{c5.4}), we know that there exists a prime
element $p_1$ in $\mathcal{D}(x)$ and a non-zero natural number $n_1$
such that:
$$
x=p^{n_1}_1a_1
$$
with $a_1\in A\setminus\{0\}$ and $\gcd(a_1,p_1)=1$. If $a_1=v_1\in
U(A)$, then $x=v_1p^{n_1}_1$ 
and the property is verified. If $a_1\not\in U(A)$, then from the
corollary (\ref{c5.4}), we can find a prime
element $p_2$ in $\mathcal{D}(a_1)$ and a non-zero natural number $n_2$
such that:
$$
a_1=p^{n_2}_2a_2
$$
and so:
$$
x=p^{n_1}_1p^{n_2}_2a_2
$$
with $a_2\in A\setminus\{0\}$, $\gcd(a_2,p_2)=1$ and $p_1\neq p_2$. 
If $a_2=v_2\in U(A)$, then $x=v_2p^{n_1}_1p^{n_2}_2$ 
and the property is verified. If $a_2\not\in U(A)$, we follow the same
steps than above. Thus, we get a sequence $(a_i)$ of elements of
$A\setminus\{0\}$ such that ($i\in\mathbb{N}^{\star}$, 
$p_i$ which is prime of $A$ and $n_i\in\mathbb{N}^{\star}$):
$$
a_{i+1}=p^{n_i}_ia_i
$$
and so:
$$
x=p^{n_1}_1\ldots p^{n_i}_ia_i
$$
with $p_m\neq p_j$ for $m\neq j$ ($m,j\in\llbracket 1,i\rrbracket$).
Since $|\mathcal{D}(x)|$ is finite, the sequence 
$(p^{n_i}_i)$ is finite. Or, the decomposition of $x$ as
$p^{n_1}_1\ldots p^{n_i}_ia_i$ is achieved when $a_i\in U(A)$.
It follows that there exists
$k\in\mathbb{N}$ such that $a_k\in U(A)$. Setting $a_k=v\in U(A)$, it results
that:
$$
x=vp^{n_1}_1\ldots p^{n_k}_k
$$
Afterwards, let prove that the decomposition of $x$ as $vp^{n_1}_1\ldots p^{n_k}_k$
with $u\in U(A)$, $p_i$ which is prime for all $i\in\llbracket
1,k\rrbracket$ and $n_1,\ldots,n_k\in\mathbb{N}^{\star}$,
is unique. Let consider two decompositions of $x$:
$$
x=vp^{n_1}_1\ldots p^{n_k}_k=wp^{m_1}_1\ldots p^{m_k}_k
$$
with $v,w\in U(A)$, $p_i$ which is prime for all $i\in\llbracket
1,k\rrbracket$ and
$n_1,\ldots,n_k,m_1,\ldots,m_k\in\mathbb{N}^{\star}$.
\\[0.1in]
Since $v\in U(A)$, we have:
$$
p^{n_1}_1\ldots p^{n_k}_k=wv^{-1}p^{m_1}_1\ldots p^{m_k}_k
$$
Since $p_i$ for all $i\in\llbracket 1,k\rrbracket$ cannot divide
$wv^{-1}\in U(A)$, it remains only one possibility that is to say
$wv^{-1}=1$ and so $w=v$. It implies that:
$$
p^{n_1}_1\ldots p^{n_k}_k=p^{m_1}_1\ldots p^{m_k}_k
$$
Let assume absurdly that $n_1\neq m_1$ say $n_1<m_1$. Since $p_1$ is
regular (see the remark (\ref{r4.10})), we have ($n_1<m_1$):
$$
\prod^k_{i\neq 1}p^{n_i}_i=p^{m_1-n_1}_1\prod^k_{i\neq 1}p^{n_i}_i
$$
Since $\gcd(p_i,p_j)=1$ for $i\neq j$ with $i,j\in\llbracket
1,k\rrbracket$, no factor of $\prod^k_{i\neq 1}p^{n_i}_i$ divides
$p^{m_1-n_1}_1$. It remains only one possibility that is to say
$p^{m_1-n_1}_1=1$ and so $m_1=n_1$. Following this reasoning for
every $i\in\llbracket 1,k\rrbracket$, it can be shown that $m_i=n_i$
for all $i\in\llbracket 1,k\rrbracket$. Therefore, the decomposition
of $x$ as $up^{n_1}_1\ldots p^{n_k}_k$
with $u\in U(A)$, $p_1,\ldots,p_k$ which are prime such that $p_i\neq p_j$ for
$i\neq j$ such that $i,j\in\llbracket 1,k\rrbracket$ and
$n_1,\ldots,n_k\in\mathbb{N}^{\star}$, 
is unique.
\end{proof}

\section{{Set operations on ideals of a principal entire ring and
 divisibility}}\label{s6}

\begin{theorem}\label{t6.1}
Let $a,b$ be two non-zero elements of the entire principal ring
$A$. If $\lcm(a,b)$ exists, then:
$$
aA\cap bA=\lcm(a,b)A
$$
\end{theorem}

\begin{proof}
We assume that $\lcm(a,b)$ exists.
\\[0.1in]
Notice that since $a|\lcm(a,b)$ and $b|\lcm(a,b)$, we have
$\lcm(a,b)A\subseteq aA$ and $\lcm(a,b)A\subseteq bA$.\\
So, $\lcm(a,b)A\subseteq aA\cap bA$.
\\[0.1in]
Since $A$ is a principal ring and since $aA\cap bA$ is an ideal, there
exists an element $m$ of $A$ such that $aA\cap bA=mA$. The element $m$
is a generator of $aA\cap bA$ (notice that it is not unique since $-m$
is also a generator of $aA\cap bA$). Let $n\in A$ be a common multiple of $a,b$.
Then, $a|n$ which implies that $nA\subseteq aA$ and $b|n$ which
implies that $nA\subseteq bA$. So, $nA\subseteq aA\cap bA$ or
equivalently $nA\subseteq mA$. It means that $m|n$. Since $n$ is
arbitrary common multiple of $a,b$ and since any ideal of $A$ is
stable by multiplication by $-1$, we deduce that $mA=\lcm(a,b)A$. It
results that $aA\cap bA=\lcm(a,b)A$.
\end{proof}

\begin{remark}\label{r6.2}
If $\lcm(a,b)$ exists such that $\lcm(a,b)\neq a$ and if
$\lcm(a,b)\neq b$, the intersection 
$aA\cap bA$ of the two ideals $aA$ and $bA$ with 
$a,b\in A$ does not contain $a$ and $b$ since the generator of $aA\cap
bA$ is $\lcm(a,b)$. Therefore, $aA\cap bA$ cannot be the smallest
ideal which is generated by the subset $\{a,b\}$.
\end{remark}

\begin{theorem}\label{t6.3}
Let $a,b$ be two non-zero elements of the entire principal ring
$A$. If $\gcd(a,b)$ exists, then:
$$
aA+bA=\gcd(a,b)A
$$
\end{theorem}

\begin{proof}
We assume that $\gcd(a,b)$ exists.
\\[0.1in]
Before giving the details of the proof, notice that since
$\gcd(a,b)|a$ and $\gcd(a,b)|b$, by linearity of the 
relation of divisibility defined on $A$, we have $\gcd(a,b)|ax+by$ for
any $x,y\in A$. 
Since $aA+bA$ is an ideal and so since $ax+by\in aA+bA$ for any
$x,y\in A$, we deduce that $aA+bA\subseteq\gcd(a,b)A$. 
\\[0.1in]
Since $A$ is a principal ring and $aA+bA$ is an ideal, there exists an
element $g$ of $A$ such that $aA+bA=gA$. The element $g$ of $A$ is a
generator of $aA+bA$ (notice that $g$ is not unique since its
symmetric $-g$ is also a generator of $aA+bA$ and notice that
$g\in\mathcal{D}(a,b)$). Since $g\in aA+bA$, there
exist two elements $x,y$ of $A$ such that $g=ax+by$. Let $d\in A$ be a
common divisor of $a,b$. Then, by linearity of the relation of
divisibility defined on $A$, $d|g$. Since $d$ is any common divisor of $a,b$ and
since any ideal of $A$ is stable by multiplication by $-1$, we deduce that
$gA=\gcd(a,b)A$. It results that $aA+bA=\gcd(a,b)A$.
\\[0.1in]
In particular, if $\mathcal{D}(a,b)=U(A)$, then by convention
$\gcd(a,b)=1$ and the ideal $aA+bA$ of $A$ is generated by an element
$g\in U(A)$ since $g\in\mathcal{D}(a,b)=U(A)$. From the lemma (\ref{l2.2}),
since $\gcd(a,b)=1$, it means that $aA+bA=A=\gcd(a,b)A$. 
\end{proof}

\begin{remark}\label{r6.4}
Provided $\gcd(a,b)$ and $\lcm(a,b)$ exist, since
$\gcd(a,b)|\lcm(a,b)$, we have the inclusion $\lcm(a,b)A\subseteq 
\gcd(a,b)A$. It is compatible with the fact that: 
$$
aA\cap bA\subseteq
aA\cup bA\subseteq aA+bA
$$
Notice that $aA\cup bA$ is not always an ideal.
\\[0.1in]
Since the ideal $aA+bA$ contains $a$ and $b$ ($a=a1+b0$ and $b=a0+b1$), then
the ideal $aA+bA$ is contained in any ideal which contains $a$ and
$b$. Therefore, $aA+bA$ is the smallest ideal which is
generated by the subset $\{a,b\}$ of $A$.
\end{remark}

\begin{corollary}\label{c6.5}
If $\gcd(a,b)$ exists, then there exist two elements $x,y$ of $A$ such that:
$$
\gcd(a,b)=ax+by
$$
\end{corollary}

\begin{proof}
We assume that $\gcd(a,b)$ exists.
\\[0.1in]
We know that $\gcd(a,b)\in aA+bA$. It results that there exist two
elements $x,y$ of $A$ such that $\gcd(a,b)=ax+by$.
\end{proof}

\begin{remark}\label{r6.6}
Thus, if $a$ and $b$ of $A$ are relatively primes, then there 
exists $(x,y)\in A^2$ such that $1=ax+by$.
\end{remark}

\section{{The Bezout identity and the Euclid's lemma in a 
principal entire ring}}\label{s7}

\begin{theorem}[Generalization of the Bezout theorem]\label{t7.1}
Let $a,b$ be two elements of $A$.
\\[0.1in]
$\mathcal{D}(a,b)=U(A)$ if, and only if, $aA+bA=A$.
\end{theorem}

\begin{proof}
If $\mathcal{D}(a,b)=U(A)$, then $\gcd(a,b)=1$ which implies that 
$aA+bA=A$. Reciprocally, if $aA+bA=A$, since $1\in A$, there exist two
elements $x,y$ of $A$ such that $1=ax+by$. Let $d\in
\mathcal{D}(a,b)$. By linearity of the relation of divisibility
defined on $A$, $d|ax+by$ and so $d|1$. Accordingly, $d\in
U(A)$. Therefore, since $d$ is any element of $\mathcal{D}(a,b)$, we
have $\mathcal{D}(a,b)\subseteq U(A)$ meaning that 
$\mathcal{D}(a,b)=U(A)$. It completes the proof that
$\mathcal{D}(a,b)=U(A)$ if, and only if, $aA+bA=A$.
\end{proof}

\begin{lemma}[Euclid's lemma]\label{l7.2}
Let $a,b,c$ be three elements of $A$. If
$\gcd(a,b)=1$ and if $a$ divides $bc$, then $a$ divides $c$. 
\end{lemma}

\begin{proof}
If $\gcd(a,b)=1$, we know that there
exist two elements $x,y$ of $A$ such that $ax+by=1$. So, for $c\in A$,
we have $c=axc+byc$ which gives
$c=acx+bcy$. The element $a$ divides $acx$. Besides, we assume
that the element $a$ of $A$ divides $bc$ and so $a$ divides
$bcy$. Accordingly, by linearity of the relation of divisibility, $a|c$.
\end{proof}

\begin{remark}\label{r7.3}
If $\gcd(a,b)=1$ and if $a|bc$ with $c\in U(A)$, then from the
generalization of the Euclid's lemma (\ref{l7.2}), $a|c$ and so $a$ should
belong to $U(A)$. In such 
a case, we have also $a|b$ since $a\in U(A)$. 
\end{remark}

\section{{Some arithmetic properties on the set of 
ideals of a principal entire ring}}\label{s8}

\begin{property}\label{p8.1}
If $a,b$ are two elements of $A$ such that $\gcd(a,b)=1$, then the common  
multiples of $a,b$ are multiples of $ab$. In particular, provided
$\gcd(a,b)=1$ and provided $\lcm(a,b)$ exists, we have: 
$$
\lcm(a,b)A=abA
$$
\end{property}

\begin{proof}
We assume that $\lcm(a,b)$ exists.
\\[0.1in]
It is obvious that the multiples of $ab$ are also common multiples of
$a,b$. Then, $abA\subseteq aA\cap bA$ or equivalently
$abA\subseteq\lcm(a,b)A$. 
\\[0.1in]
Reciprocally, let assume that $m\in aA\cap bA$. Then, we have $m=bc$
with $c\in A$. As $a|m$ (namely, $a|bc$) and $\gcd(a,b)=1$,
from the Euclid's lemma (\ref{l7.2}), $a|c$. So,
there exists an element $x$ of $A$ such that $c=ax$. Therefore,
$m=abx$ with $x\in A$. Accordingly, $ab|m$. Since $m$ is any element
of the subset   
$aA\cap bA$ of $A$, it implies that 
$aA\cap bA\subseteq abA$ or equivalently $\lcm(a,b)A\subseteq abA$.
\\[0.1in]
Therefore, we conclude that $\lcm(a,b)A=abA$.
\end{proof}

\begin{corollary}\label{c8.2}
Let $n\in\mathbb{N}^{\star}$. 
If elements $a_1,\ldots,a_n$ of $A$ whose $\gcd$ is equal to $1$, divides 
an element $m$ of $A$, then their product $a_1\ldots a_n$ divides $m$. 
\end{corollary}

\begin{corollary}\label{c8.3}
Let $a,b$ be two non-zero elements of $A$ such that their $\gcd$ exists,
$a=ga'$, $b=gb'$ with $g=\gcd(a,b)$ which is multiplicatively
regular. We assume that $\lcm(a,b)$ exists. Then, the common  
multiples of $a,b$ are multiples of 
$ga'b'$. In particular, we have:
$$
\lcm(a,b)\gcd(a,b)A=abA
$$
\end{corollary}

\begin{proof}
We assume that $\gcd(a,b)$ exists and is regular. We assume also that
$\lcm(a,b)$ exists.
\\[0.1in]
Notice that $\gcd(a',b')=1$. Indeed, denoting $\gcd(a,b)=g$, there exist
two elements $x,y$ of $A$ such that $g=ax+by$. Since $a=ga'$ and $b=gb'$, it
comes that $g=ga'x+gb'y$. Since $g$ is multiplicatively regular, after
simplification, we get $1=a'x+b'y$. So, the ideal $a'A+b'A$ contains
$1$ meaning that $a'A+b'A=A$ (see the lemma (\ref{l2.2})). From the
generalization of the Bezout theorem (\ref{t7.1}), we have 
$\mathcal{D}(a',b')=U(A)$. It means that $\gcd(a',b')=1$.
\\[0.1in]
It is obvious that $ga'b'=ab'=a'b$ is a multiple of $a,b$. Then
$ga'b'A\subseteq aA\cap bA$ and so $ga'b'A\subseteq\lcm(a,b)A$. 
\\[0.1in]
Reciprocally, let $m$ be a common multiple of $a,b$. Then by
transitivity of the relation of divisibility, $m$ is a
multiple of $g$. So, we have $m=gm'$, $m=ac$ and $m=bd$
with $m',c,d\in A$. Since $a=ga'$ and $b=gb'$, it gives $gm'=ga'c$ and 
$gm'=gb'd$. Whence since $g$ is multiplicatively regular, we obtain
$m'=a'c$ and $m'=b'd$ with $c,d\in A$. We deduce that $a'|m'$ and $b'|m'$. 
Since $a'|m'$, $b'|m'$ and since $\gcd(a',b')=1$, 
it results that $a'b'|m'$. Therefore, $ga'b'|gm'$ and so
$ga'b'|m$. Since $m$ is any multiple
of $a,b$, we deduce that $aA\cap bA\subseteq ga'b'A$ and so 
$\lcm(a,b)A\subseteq ga'b'A$.
\\[0.1in]
Therefore, we conclude that $\lcm(a,b)A=ga'b'A$ or
equivalently $\lcm(a,b)=ab'A=a'bA$. It results that
$\lcm(a,b)\gcd(a,b)A=abA$.
\end{proof}

\section{{Maximal ideals in a principal entire ring}}\label{s9}

A prime ideal\cite{sg} in $A$\index{prime ideal} is an ideal
$\mathfrak{p}\neq A$ such 
that $A/\mathfrak{p}$ is entire. Equivalently, we could say that it is 
an ideal $\mathfrak{p}\neq A$ such that, whenever $x,y\in A$ and
$xy\in\mathfrak{p}$, then $x\in\mathfrak{p}$ or $y\in\mathfrak{p}$. 

\begin{property}\label{p9.1}
Let $z$ be an element of $A$ such that $N(z)=z$.
\\[0.1in]
$A/zA$ is entire if, and only if, $z$ is prime or null.
\end{property}

\begin{proof}
If $z=0$, then $A/zA$ is equal to $A$ which is entire. In the
following, we assume that $z$ is not null.
\\[0.1in]
Let assume that $z=N(z)\neq 0$ is a prime element of $A$ (see above for a
definition). Then, $zA$ is an ideal of $A$ which is not equal to $A$
since $z$ is not invertible. Let $x,y$ be two elements of $A$ such 
that $z|xy$. Then either $z$ divides both $x,y$ 
or else $z$ and one of the elements $x,y$ say $x$ are
relatively primes. In the second case, from the Euclid's lemma, it follows
that $z$ divides the other element namely $y$ among the elements $x,y$
of $A$. To sum up, if $z$ is a prime element of $A$ and if $z|xy$ with
$x,y\in A$, then $z|x$ or $z|y$. It is equivalent to say that $zA\neq A$ is
an ideal such that whenever $x,y\in A$ and
$xy\in zA$, then $x\in zA$ or $y\in zA$. From the definition of a
prime ideal, it results that $A/zA$ is entire.
\\[0.1in]
Afterwards, let assume that $z\neq 0$ is not a prime element of $A$. If
$z\in U(A)$, then $A/zA$ is reduced to one residue class of elements of
$A$ which is equal to the zero class. In this case, $A/zA$ is clearly
not entire. Let assume that $z\not\in U(A)$. Since $z$ is not prime
and $z\not\in U(A)$, $\mathcal{D}(z)$ is not reduced to $U(A)$. So,
there exist two elements $x,y\in A$ such that $z=xy$ (with at least one
of the elements $x,y$, which does not belong to $U(A)$). If one of the
elements $x,y$ belongs to $U(A)$, then $\mathcal{D}(z)=U(A)\cup
zU(A)$. Since $z\not\in U(A)$ and $z\neq 0$, from the definition of a
prime element of $A$, it would mean that $z$ is prime. What it is
impossible. So, we have necessarily $z=xy$ with $x,y\in A$ such that
$x,y\not\in U(A)$. Then, the residue class
$\hat{x},\hat{y}$ of $x,y$ in $A/zA$ 
are non-zero and their product is zero. It means that $A/zA$ has
divisors of zero and so $A/zA$ is not
entire. Thus, we proved that if $z$ is not a prime element of $A$,
then $A/zA$ is not entire. It is equivalent to say that if $A/zA$ is
entire, then $N(z)$ is prime.
\\[0.1in]
We conclude that $A/zA$ is entire if, and only if, $z$ is prime or
null.
\end{proof}

\begin{property}\label{p9.2}
Let $z$ be an element of $A$.
\\[0.1in]
If $z$ is prime, then $zA$ is maximal.
\end{property}

\begin{proof}
Let $z$ be a prime element of $A$. Let assume that there exists $w\in
A$ such that $zA\subseteq wA$. Then $w|z$. Since $z$ is prime (see the
definition of a prime element of $A$), then there exists $u\in U(A)$
such that $z=wu$. It results that $zA=wuA=wA$.
\end{proof}

\begin{corollary}\label{c9.3}
Let $z$ be an element of $A$ such that $N(z)=z$.
\\[0.1in]
$A/zA$ is a field if, and only if, $z$ is prime.
\end{corollary}

\begin{proof}
Let $z$ be an element of $A$ such that $N(z)=z$.
\\[0.1in]
We know that the ideal $zA$ is maximal if, and only if, $A/zA$ is a
field (see p. 93 of \cite{sg}). So, if $z$
is prime, then from the property above, the ideal $zA$ is maximal. It
results that if $z$ is prime, then $A/zA$ is a field.
\\[0.1in]
Reciprocally, if $A/zA$ is a field with $z=N(z)$, then $z\neq 0$ and
$z\not\in U(A)$ ($A$ is not a field by assumption and
$A/A=\{\hat{0}\}$ is not also a field). Let consider a non-zero
element $x$ whose residue class 
$\hat{x}$ in $A/zA$ is invertible in $A/zA$. Notice that $x\neq zq$ with
$q\in A$ and in particular $x\neq zu$ with $u\in U(A)$. Since $A/zA$ is a field,
there exists an element $y\in A$ such that
$\hat{x}\,\hat{y}=\hat{1}$. It means that $1=xy+zw$ with $w\in
A$. Then, it comes that for any $a\in A$, we have
$a=xya+zwa$. Accordingly, since $a$ is arbitrary in $A$, 
it results that $A\subseteq xA+zA$ which gives $xA+zA=A$. Or, from the
generalization of the Bezout theorem (\ref{t7.1}), 
$xA+zA=A$ is equivalent to $\mathcal{D}(z,x)=U(A)$. Since $x$ is any non-zero
element of $A$ whose residue class
$\hat{x}$ in $A/zA$ is invertible in $A/zA$ (namely $x\neq zq$ with
$q\in A$ and in particular $x\neq zu$ with $u\in U(A)$), only elements
of $U(A)$ and elements of $zU(A)$ in $A$, divide $z$. So, we get
$\mathcal{D}(z)=U(A)\cup zU(A)$. Since $z\neq 0$ and $z\not\in U(A)$
if $A/zA$ is a field with $z=N(z)$, we deduce that if $A/zA$ is a
field, then $z$ is prime. 
\\[0.1in]
We conclude that when $z=N(z)$, $A/zA$ is a field if, and only if,
$z$ is prime.
\end{proof}

\section{{Examples of ring extensions of entire rings}}\label{s10}

In this section, the element $i$ which verifies the polynomial
equation $z^2+1=0$ in $\mathbb{C}$, is also written in its exponential
form as:
$$
i=e^{\frac{i\pi}{2}}
$$
Thus, the equation $i^2+1=0$ can be rewritten as:
$$
e^{i\pi}+1=0
$$
Since $i^{\star}=-i$ (see the remark (\ref{r4.41})), using the
definition (\ref{d4.25}), we have necessarily $i=i'$ (recall that $i$
is defined as a unit in $A[S]$ (see the property (\ref{p4.40})) and so
is invertible). From the remark
(\ref{r4.35}), using again the remark (\ref{r4.41}), it follows that:
$$
i^{\star}=i^{-1}N(i)^2=i^{-1}
$$ 
Therefore:
$$
(e^{\frac{i\pi}{2}})^{\star}=(e^{\frac{i\pi}{2}})^{-1}
$$
$$
(e^{\frac{i\pi}{2}})^{\star}=e^{-\frac{i\pi}{2}}
$$
Since $i^{\star}=-i$, it gives:
$$
(e^{\frac{i\pi}{2}})^{\star}=e^{-\frac{i\pi}{2}}=-i
$$
In the following, for all $n\in\mathbb{N}^{\star}$, we set:
$$
e_{n,0}=e_{n,n}=-i^2=-e^{i\pi}=1
$$
It comes that:
$$
e^2_{n,0}=e^2_{n,n}=i^4=e^{2i\pi}=1
$$
Let $n\in\mathbb{N}^{\star}$ and let $U_n$ the subset of the $n$th-root
of unity in $\mathbb{C}$:
$$
U_n=\{e_{n,k}\,:\,e^n_{n,k}=e_{n,0},\,k=0,1,\ldots,n-1\}
$$
with ($k=0,1,\ldots,n-1$ and $n\in\mathbb{N}^{\star}$):
$$
e_{n,k}=(-e_{n,0})^{\frac{2k}{n}}=(e^{i\pi})^{\frac{2k}{n}}=e^{\frac{2ik\pi}{n}}
$$
and for $k,l\in\{0,1,\ldots,n-1\}$:
$$
e_{n,k}=e_{n,l}\,\Leftrightarrow\,k=l
$$
Notice that the formula of $e_{n,k}$ works also for $k=n$:
$$
e_{n,n}=e^{2i\pi}=1=e_{n,0}
$$
Notice also that the square root function defined on
$\mathbb{C}$ is the function which associates at least one complex
number $w$ to any complex number $z$ such that $w^2=z$. It   
maps $U_n$ onto $U_{2n}\cup e^{i\pi}U_{2n}$ with $n\in\mathbb{N}^{\star}$
since if ($k=0,1,\ldots,n-1$ with $n\in\mathbb{N}^{\star}$):
$$
w^2=e^{\frac{2ik\pi}{n}}=e^{\frac{2ik\pi}{n}}
$$
then ($k=0,1,\ldots,n-1$ with $n\in\mathbb{N}^{\star}$):
$$
w=\pm e^{\frac{2ik\pi}{2n}}=\pm e_{2n,k}
$$
The ambiguity of sign comes from the fact that we can write
$z=ze^{2i\pi}$. So, we can replace $z$ by $ze^{2i\pi}$ in the equality
$w^2=z$. It may change $w$ into $-w$. In order to have a single value, we cut
the complex plane where the square root function is multi-valued. The
set of points (images of complex numbers in the complex plane) where
the square root function keeps a constant sign is called a branch (or a sheet).
The result is to make the square root function
uniform (within a branch). The corresponding value of the square root
function at a complex number $z$ in the cut complex plane is the chosen
determination of the square root of a complex number $z$. When the argument of
$z$ (namely the angle coordinate of $z$ when $z$ is described by its polar
coordinates in the complex plane) belongs to the interval
$]-\pi;\pi]$, the chosen determination of 
the square root of a complex number $z$ is called
its principal square root denoted $\sqrt{z}=z^{\frac{1}{2}}$. It is usual
to take the branch cut in the complex plane as the non-positive
part of the real axis in the complex plane. It stems from the fact
that we can associate two purely imaginary complex
numbers which are the 
solutions of a polynomial equation as $z^2+x=0$ of unknown $z$ in
$\mathbb{C}$, to 
any strictly negative real number $-x$ with $x>0$. Each time we go through
the branch cut, the square root function takes a multiplicative global
sign $-1$. So, the square root function has a discontinuity near the
branch cut.
\\[0.1in]
For instance, the principal square root of
$e_{2,1}=e^{i\pi}=-e_{2,0}=-1$ is given by:
$$
i=e_{4,1}=\sqrt{e_{2,1}}=\sqrt{-e_{2,0}}=e_{2,0}\sqrt{-1}
$$
where we used the fact that $\sqrt{e_{n,0}}=e_{n,0}=1$ for all
$n\in\mathbb{N}^{\star}$.
\\[0.1in]
More generally, the principal square root of $e_{n,k}$ for
$0\leq k\leq \frac{n}{2}$ with $n\in\mathbb{N}^{\star}$ is given by:
$$
\sqrt{e_{n,k}}=e_{2n,k}
$$
Therefore, the principal square root function defined on the cut
complex plane, maps $U_n$ onto
$U_{2n}$ with $n\in\mathbb{N}^{\star}$.
\\[0.1in]
Otherwise, since ($k=0,1,\ldots,n-1$):
$$
e^{2i\pi}=1\,\Rightarrow\,(e^{2i\pi})^k=e^{2ik\pi}=1
$$
we can observe that ($k=0,1,\ldots,n-1$):
$$
(e^{\frac{2ik\pi}{n}})^{n-1}=e^{\frac{2ik(n-1)\pi}{n}}
=e^{-\frac{2ik\pi}{n}}=e^{\frac{2i(n-k)\pi}{n}}
$$
So ($k=0,1,\ldots,n-1$):
$$
e^n_{n,k}=e_{n,0}\,\Leftrightarrow\,e_{n,k}e^{n-1}_{n,k}=e_{n,0}
\Leftrightarrow\,e_{n,k}e_{n,n-k}=e_{n,0}
$$
and ($k=0,1,\ldots,n-1$):
$$
e^n_{n,k}=e_{n,0}\,\Leftrightarrow\,e^{n-1}_{n,k}e_{n,k}=e_{n,0}
\Leftrightarrow\,e_{n,n-k}e_{n,k}=e_{n,0}
$$
Therefore ($k=0,1,\ldots,n-1$):
$$
e_{n,k}e_{n,n-k}=e_{n,n-k}e_{n,k}=e_{n,0}
$$
It results that each $e_{n,k}$ with $n\in\mathbb{N}^{\star}$ and
$k=0,1,\ldots,n-1$ is invertible. Consequently, $N(e_{n,k})\neq 0$ for
all $k\in\{0,1\ldots,n-1\}$ with $n\in\mathbb{N}^{\star}$. Regarding
the Euclid division
of the product $km$ of two integers $k,m\in\{0,1,\ldots,n-1\}$ by $n$:
$$
km=nq+r\,\,\,\,\mathrm{with}\,\,0\leq r<n
$$
where $q=\lfloor\frac{km}{n}\rfloor$, we can remark that
($k,m=0,1,\ldots,n-1$ and $r$ the remainder of the Euclid division of
$km$ by $n$):
$$
(e^{\frac{2ik\pi}{n}})^m=e^{\frac{2ikm\pi}{n}}=e^{2iq\pi+\frac{2ir\pi}{n}}
=e^{\frac{2ir\pi}{n}}
$$
and so we have ($k,m=0,1,\ldots,n-1$ and $r$ the remainder of the
Euclid division of $km$ by $n$):
$$
e^m_{n,k}=e_{n,r}
$$
Or ($k=0,1,\ldots,n-1$):
$$
N(e^n_{n,k})=N(e_{n,0})=N(1)=1
$$
Using the property (\ref{p4.16}), it comes that:
$$
N(e_{n,k})^n=1
$$
$$
N(e_{n,k})^n-1=0
$$
$$
(N(e_{n,k})-1)(1+\ldots+N(e_{n,k})^{n-1})=0
$$
Since in the sum $1+\ldots+N(e_{n,k})^{n-1}$, each term can be written
as $N(e_{n,r})\neq 0$ with $0\leq r<n$, it implies that ($k=0,1,\ldots,n-1$):
$$
N(e_{n,k})=1
$$
It results also that ($k=0,1,\ldots,n-1$):
$$
u(e_{n,k})=e_{n,k}
$$
Since $e_{n,k}$ is invertible, using the defintion (\ref{d4.25}), we
have ($k=0,1,\ldots,n-1$):
$$
(e_{n,k})^{\star}=e'_{n,k}=e^{-1}_{n,k}N(e_{n,k})=e^{-1}_{n,k}=e_{n,n-k}
$$ 
or equivalently ($k=0,1,\ldots,n-1$):
$$
(e^{\frac{2ik\pi}{n}})^{\star}=(e^{\frac{2ik\pi}{n}})^{-1}
=e^{-\frac{2ik\pi}{n}}=e^{\frac{2i(n-k)\pi}{n}}
$$
Thus, we have ($k=0,1,\ldots,n-1$):
$$
N(e_{n,k})^2=e_{n,k}(e_{n,k})^{\star}=(e_{n,k})^{\star}e_{n,k}=e_{n,0}=1
$$
Besides, since $n-k\in\llbracket 1,n-1\rrbracket$ for
$k=1,\ldots,n-1$ and since $e^{\star}_{n,0}=e_{n,0}$, 
$e^{\star}_{n,k}\in U_n$ for all $k\in\llbracket 0,n-1\rrbracket$ and
the star operation which is well 
defined on $U_n$, maps $U_n$ onto $U_n$.
\\[0.1in]
It is straightforward to verify that $U_n$ is a multiplicative cyclic group of
order $n$. In particular, a generator of $U_n$ is given by:
$$
e_{n,1}=e^{\frac{2i\pi}{n}}
$$
and we have ($k=0,1,\ldots,n-1$):
$$
e_{n,k}=e^k_{n,1}
$$
Notice that if $n$ is even (but not zero), then $n/2$ is a non-zero
natural number and we have ($n\in 2\mathbb{N}^{\star}$):
$$
e_{n,n/2}=-1=-e_{n,0}
$$ 
For $n$ odd, it is impossible that an element among the elements
$e_{n,k}$ s with $k\in\llbracket 0,n-1\rrbracket$, of $U_n$, be equal to $-1$.
\\[0.1in]
Moreover, we can define elements $e_{n,m}$ for all
$m\in\mathbb{Z}$. Indeed, regarding the Euclid division of
$\mathrm{abs}(m)$ by $n$: 
$$
\mathrm{abs}(m)=an+k\,\,\,\,\mathrm{with}\,\,0\leq k<n
$$
with $a=\lfloor\frac{\mathrm{abs}(m)}{n}\rfloor$, we have:
$$
e_{n,\mathrm{abs}(m)}=e^{\frac{2im\pi}{n}}=e^{2ia\pi+\frac{2ik\pi}{n}}=e^{\frac{2ik\pi}{n}}=e_{n,k}
$$
and we have:
$$
e_{n,-\mathrm{abs}(m)}=e^{-1}_{n,\mathrm{abs}(m)}=e^{-1}_{n,k}=e_{n,n-k}
$$
In particular, for $m,l\in\mathbb{Z}$:
$$
e_{n,\mathrm{abs}(m)}=e_{n,\mathrm{abs}(l)}\,\Leftrightarrow\,
\mathrm{abs}(m)=\mathrm{abs}(l)+qn\,\,\,\,\mathrm{with}\,\,q\in\mathbb{Z}
$$
Or, we have ($k=0,1,\ldots,n-1$ and $n\in\mathbb{N}^{\star}$):
$$
e_{n,k}=(e_{n,k})^{\star}\,\Leftrightarrow\,
e^{\frac{2ik\pi}{n}}=e^{-\frac{2ik\pi}{n}}
\,\Leftrightarrow\,e^{\frac{4ik\pi}{n}}=1\,
\Leftrightarrow\,e_{n,2k}=e_{n,0}\,\,\,\,\mathrm{or}\,\,\,\,e_{n,2k}=e_{n,n}
$$  
So, since $2k\in\{0,2,\ldots,2n-2\}$ for $k\in\{0,1,\ldots,n-1\}$ with
$n\in\mathbb{N}^{\star}$, 
we have either $2k=0$ which gives $k=0$ or $2k=n$ which gives $k=n/2$.
The case $k=n/2$ is only possible if $n$ is even
(but not zero). So, provided $n$ is even, the set $U_n$ contains two
real numbers namely $e_{n,0}=1$ and $e_{n,n/2}=-e_{n,0}=-1$.
\\[0.1in] 
In conclusion, we get:
$$
U_n\cap\mathbb{R}=\left\{\begin{array}{ccc}
\{1\} & \mathrm{if} & n\equiv 1\pmod 2
\\
\{1,-1\} & \mathrm{if} & n\equiv 0\pmod 2
\end{array}
\right.
$$
In the following, we denote $I_n$ the integer interval:
$$
I_n=\llbracket 0,n-1\rrbracket
$$
Let $\mathbb{Z}[S_n]$ the subring of $\mathbb{C}$ generated by $S_n$
over the subring $\mathbb{Z}$ of integers of $\mathbb{C}$ with $S_n$
given by $S_n=\mathcal{G}_n\setminus\{e\in\mathcal{G}_n\,:\,e\in\mathbb{Z}\}$
where $\mathcal{G}_n$ is a maximal 
family of linearly independent elements of $U_n$ over
$\mathbb{C}$:
$$
\mathbb{Z}[S_n]=\{a_0e_{n,0}+a_1e_{n,1}+\ldots+a_{n-1}e_{n,n-1}
\,:\,a_k\in\mathbb{Z},\,
k=0,1,\ldots,n-1\}
$$
We can notice that for $n=1,2$, $U_{n}\subseteq \{1,-1\}$ and so
$\mathbb{Z}[S_n]=\mathbb{Z}$. 
For $n\geq 3$, since ($k\in I_n$):
$$
e^n_{n,k}=e_{n,0}=1
$$
and:
$$
e_{n,0}+e_{n,1}+\ldots+e_{n,n-1}=0
$$
the family $\{e_{n,k}\}_{k\in I_n}$ with $n\in\mathbb{N}^{\star}$
which generates $\mathbb{Z}[S_n]$, is not free. It results that
for $n\geq 3$, $\mathbb{Z}[S_n]$ is generated by the family
$\{e_{n,0},e_{n,1},\ldots,e_{n,n-2}\}$ if $n\equiv 1\pmod 2$ and
$\mathbb{Z}[S_n]$ is generated by the family 
$\{e_{n,0},e_{n,1},\ldots,e_{n,n-2}\}\setminus\{e_{n,n/2}\}$ if $n\equiv 0\pmod 2$.
\\[0.1in]
For instance, 
for $n=3$, setting $j=e_{3,1}=e^{\frac{2i\pi}{3}}$, we have ($j^3=1$):
$$
1+j+j^2=0
$$
and:
$$
j^2=j^{\star}=-e_{3,0}-j
$$
Notice that:
$$
jj^{\star}=jj^2=j^3=1
$$
The subring $\mathbb{Z}[S_3]$ of $\mathbb{C}$ is generated by
$S_3=\mathcal{G}_3\setminus\{e_{3,0}\}=\{j\}$ 
where $\mathcal{G}_3=\{e_{3,0},j\}$
with $e_{3,0}=1$: 
$$
\mathbb{Z}[S_3]=\mathbb{Z}[j]=\{ae_{3,0}+bj\,:\,a,b\in\mathbb{Z}\}
$$
Since $\{e_{3,0},j\}$ is free and is maximal in $U_3$, $\mathbb{Z}[S_3]$ has
a basis namely $\{e_{3,0},j\}$.
\\[0.1in]
In this case, we have ($e^{i\pi}=-1$):
$$
U(\mathbb{Z}[j])=U_3\cup e^{i\pi}U_3=\{e_{3,0},-e_{3,0},j,-j,j^{\star},-j^{\star}\}
$$
Indeed, let $ae_{3,0}+bj\in U(\mathbb{Z}[j])$. Then, there exists
$a',b'\in\mathbb{Z}$ such that:
$$
(ae_{3,0}+bj)(a'e_{3,0}+b'j)=e_{3,0}
$$
$$
aa'e_{3,0}+(ab'+a'b)j+bb'j^2=e_{3,0}
$$
$$
aa'e_{3,0}+(ab'+a'b)j-bb'(e_{3,0}+j)=e_{3,0}
$$
$$
(aa'-bb')e_{3,0}+(ab'+a'b-bb')j=e_{3,0}
$$
So, since $\{e_{3,0},j\}$ is a basis of $\mathbb{Z}[S_3]$, we have:
$$
\left\{\begin{array}{c}
aa'-bb'=1
\\
ab'+ba'-bb'=0
\end{array}\right.
$$
The equality $ab'+ba'-bb'=0$ can be rewritten as:
$$
bb'=ab'+ba'
$$
The equation $aa'-bb'=1$ means that $\gcd(a,b)=1$. Moreover, the equation
$aa'-bb'=1$ can be rewritten like $a(a'-b')-(b-a)b'=1$ meaning that
also $\gcd(a,\mathrm{abs}(b-a))=1$. Besides the equation
$ab'+ba'-bb'=0$ can be rewritten as:
$$
b'(b-a)=a'b
$$
Since $\gcd(a,\mathrm{abs}(b-a))=1$, from the Euclid's lemma, $b-a|a'$
and $b|b'$. So, there exists an integer $k$ such that:
$$
a'=(b-a)k
$$
$$
b'=bk
$$
Using the equation $aa'-bb'=1$, it implies that:
$$
(ab-(a^2+b^2))k=1
$$
So, either $k=1$ or $k=-1$. If $k=1$, then $ab-(a^2+b^2)=1$ which is
equivalent to $a^2-ab+b^2+1=0$. This equation of degree $2$ in
variable $a$ has no solution in $\mathbb{R}$ and so in $\mathbb{Z}$
since its discriminant is $\Delta=-3b^2-4<0$. So, the case $k=1$ is
not possible. If $k=-1$, then $ab-(a^2+b^2)=-1$ which is
equivalent to $a^2-ab+b^2-1=0$. The discriminant of this equation of
degree $2$ in variable $a$ is $\Delta=4-3b^2$. Since $b$ is an
integer, either $b=-1$, either $b=0$ or $b=1$. Notice that for $b=0$
or $b=\pm 1$, $\Delta$ is strictly positive and is a perfect
square. It gives rise to the possible values for $a$:
$$
b=-1\,\Rightarrow\,a=0\,\,\,\,\mathrm{or}\,\,\,\,a=-1
$$
$$
b=0\,\Rightarrow\,a=-1\,\,\,\,\mathrm{or}\,\,\,\,a=1
$$
$$
b=1\,\Rightarrow\,a=0\,\,\,\,\mathrm{or}\,\,\,\,a=1
$$
Therefore, the elements of $U(\mathbb{Z}[j])$ are:
$$
e_{3,0},-e_{3,0},j,-j,j^{\star},-j^{\star}
$$
We can notice that $\mathbb{Z}[j]=\mathbb{Z}\oplus j\mathbb{Z}$. It is because
any element $x$ of $\mathbb{Z}[j]$ is written in an unique way as 
$x=au_{3,0}+bj$ with $a,b\in\mathbb{Z}$ and because $\mathbb{Z}\cap
j\mathbb{Z}=\{0\}$. Indeed, the existence of $a,b\in\mathbb{Z}$ stems
from the algebraic structure of $\mathbb{Z}[j]$. For proving the
uniqueness of $a,b\in\mathbb{Z}$ such that $x=au_{3,0}+bj$, let consider two
other integers $a',b'$ which verify:
$$
ae_{3,0}+bj=a'e_{3,0}+b'j
$$
Since $\{e_{3,0},j\}$ is a basis of $\mathbb{Z}[S_3]$, it gives $a=a'$
and $b=b'$ meaning that $x$ is written uniquely as 
$x=a+bj$ with $a,b\in\mathbb{Z}$. Moreover, let $a,b\in\mathbb{Z}$
such that:
$$
ae_{3,0}=bj
$$
Then, since $\{e_{3,0},j\}$ is a basis of $\mathbb{Z}[S_3]$, $a=b=0$
meaning that $\mathbb{Z}\cap j\mathbb{Z}=\{0\}$. 
So, for any subset $\mathcal{W}$ of 
$\mathbb{Z}[j]$, there exist two subsets $\mathcal{X},\mathcal{Y}$ of
$\mathbb{Z}$ such that $\mathcal{W}=\mathcal{X}\oplus
j\mathcal{Y}$. Notice that $\mathcal{X},\mathcal{Y}$ are unique since
if $\mathcal{X}'\oplus j\mathcal{Y}'=\mathcal{X}\oplus j\mathcal{Y}$,
then $\mathcal{X}'=\mathcal{X}$ and $\mathcal{Y}'=\mathcal{Y}$. Let
$\mathcal{L}$ an ideal of $\mathbb{Z}[j]$. Since
$\mathcal{L}$ is a subset of $\mathbb{Z}[j]$, there exist two subsets
$\mathcal{I},\mathcal{J}$ of $\mathbb{Z}$ such that
$\mathcal{L}=\mathcal{I}\oplus j\mathcal{J}$. Since
$\mathcal{L}$ is an additive subgroup of $\mathbb{Z}[j]$ which is
stable by multiplication, then $\mathcal{I},\mathcal{J}$ should be
additive subgroups of $\mathbb{Z}$ which are stable by multiplication. So,
$\mathcal{I},\mathcal{J}$ are ideals of $\mathbb{Z}$. Reciprocally, if
$\mathcal{L}=\mathcal{I}\oplus j\mathcal{J}$ where
$\mathcal{I},\mathcal{J}$ are ideals of $\mathbb{Z}$, it is obvious
that $\mathcal{L}$ is also an ideal of $\mathbb{Z}[j]$.
\\[0.1in]
Since in $\mathbb{Z}$, any ideal is principal, any ideal of
$\mathbb{Z}[j]$ has the form ($a,b\in\mathbb{Z}$):
$$
\mathcal{L}_{a,b}=\{ae_{3,0}x+byj\,:\,x,y\in\mathbb{Z}\}
$$
Moreover, we know that for $a,b$ in $\mathbb{Z}$, there exists
two integers $s,t$ such that:
$$
a=s\gcd(a,b)
$$ 
and:
$$
b=t\gcd(a,b)
$$
Therefore, any element of $\mathcal{L}_{a,b}$ can be expressed as
($x,y\in\mathbb{Z}$):
$$
ae_{3,0}x+byj=\gcd(a,b)(ae_{3,0}s+btj)
$$
It results that $\mathcal{L}_{a,b}$ is generated by $\gcd(a,b)$ and is
so principal. We conclude that $\mathbb{Z}[j]$ is a principal entire
subring of $\mathbb{C}$.
\\[0.1in]
The expression for the magnitude function is given by:
$$
N(ae_{3,0}+bj)=\sqrt{(ae_{3,0}+bj)(ae_{3,0}+bj^{\star})}=
\sqrt{a^2e^2_{3,0}+ab(j+j^{\star})+b^2jj^{\star}}
$$
$$
N(ae_{3,0}+bj)=\sqrt{(a^2-2ab+b^2)e_{3,0}}=\sqrt{(a-b)^2e_{3,0}}
$$
$$
N(ae_{3,0}+bj)=\mathrm{abs}(a-b)\sqrt{e_{3,0}}=\mathrm{abs}(a-b)e_{3,0}
=\mathrm{abs}(a-b)
$$
Thus, we have:
$$
\ker N=\{a(u_{3,0}+j)\,:\,a\in\mathbb{Z}\}
$$
It follows that the expression for the unit function is given by
($a\neq b$):
$$
u(ae_{3,0}+bj)=\frac{1}{\mathrm{abs}(a-b)}(a+bj)
$$
For $n=4$, we know that $i=e_{4,1}=e^{\frac{i\pi}{2}}$.
The subring $\mathbb{Z}[S_4]$ of $\mathbb{C}$ is generated by
$S_4=\mathcal{G}_4\setminus\{e_{4,0}\}=\{i\}$
where $\mathcal{G}_4=\{e_{4,0},i,e_{4,2}\}\setminus\{e_{4,2}\}$ namely
$\mathcal{G}_4=\{e_{4,0},i\}$ with $e_{4,0}=1$. Since $\{u_{4,0},i\}$ is free and is
maximal in $U_4$, 
$\mathbb{Z}[S_4]$ has a basis namely $\mathcal{G}_4=\{e_{4,0},i\}$. 
So, $\mathbb{Z}[S_4]$ is the subring of Gaussian integers:
$$
\mathbb{Z}[S_4]=\mathbb{Z}[i]=\{ae_{4,0}+bi\,:\,a,b\in\mathbb{Z}\}
$$
In this case, we have:
$$
U(\mathbb{Z}[i])=U_4=\{e_{4,0},-e_{4,0},i,-i\}
$$
Indeed, let $ae_{4,0}+bi\in U(\mathbb{Z}[i])$. Then, there exists
$a',b'\in\mathbb{Z}$ such that:
$$
(ae_{4,0}+bi)(a'e_{4,0}+b'i)=e_{4,0}
$$
$$
(aa'-bb')e_{4,0}+(ab'+ba')i=e_{4,0}
$$
So:
$$
\left\{\begin{array}{c}
aa'-bb'=1
\\
ab'+ba'=0
\end{array}\right.
$$
The first equation means that $\gcd(a,b)=1$ and so $a$ and $b$ are
relatively primes. Using this fact, the second equation which can be
rewritten $ab'=-ba'$, implies from the Euclid's lemma that $a|a'$ and
$b|b'$. So, there exists an integer $k$ in $\mathbb{Z}$ such that:
$$
a'=ak
$$
and:
$$
b'=bk
$$
From the equation $aa'-bb'=1$, it gives:
$$
k(a^2+b^2)=1
$$
So, either $k=1$ or $k=-1$. If $k=-1$, then $a^2+b^2=-1$ which is not
possible. It remains $k=1$. It gives ($k=1$):
$$
\begin{array}{ccc}
a=\pm 1 & \mathrm{and} & b=0
\\
& \mathrm{or} &
\\
a=0 & \mathrm{and} & b=\pm 1
\end{array}
$$
Therefore, the elements of $U(\mathbb{Z}[i])$ are:
$$
e_{4,0},-e_{4,0},i,-i
$$
Since $\mathbb{Z}[i]=\mathbb{Z}\oplus i\mathbb{Z}$ (it is because
any element $x$ of $\mathbb{Z}[i]$ is written in an unique way as $x=a+bi$ with
$a,b\in\mathbb{Z}$ and because $\mathbb{Z}\cap i\mathbb{Z}=\{0\}$),
for any subset $\mathcal{W}$ of 
$\mathbb{Z}[i]$, there exist two subsets $\mathcal{X},\mathcal{Y}$ of
$\mathbb{Z}$ such that $\mathcal{W}=\mathcal{X}\oplus
i\mathcal{Y}$. Notice that $\mathcal{X},\mathcal{Y}$ are unique since
if $\mathcal{X}'\oplus i\mathcal{Y}'=\mathcal{X}\oplus i\mathcal{Y}$,
then $\mathcal{X}'=\mathcal{X}$ and $\mathcal{Y}'=\mathcal{Y}$. Let
$\mathcal{L}$ an ideal of $\mathbb{Z}[i]$. Since
$\mathcal{L}$ is a subset of $\mathbb{Z}[i]$, there exist two subsets
$\mathcal{I},\mathcal{J}$ of $\mathbb{Z}$ such that
$\mathcal{L}=\mathcal{I}\oplus i\mathcal{J}$. Since
$\mathcal{L}$ is an additive subgroup of $\mathbb{Z}[i]$ which is
stable by multiplication, then $\mathcal{I},\mathcal{J}$ should be
additive subgroups of $\mathbb{Z}$ which are stable by multiplication. So,
$\mathcal{I},\mathcal{J}$ are ideals of $\mathbb{Z}$. Reciprocally, if
$\mathcal{L}=\mathcal{I}\oplus i\mathcal{J}$ where
$\mathcal{I},\mathcal{J}$ are ideals of $\mathbb{Z}$, it is obvious
that $\mathcal{L}$ is also an ideal of $\mathbb{Z}[i]$.
\\[0.1in]
Since in $\mathbb{Z}$, any ideal is principal, any ideal of
$\mathbb{Z}[i]$ has the form ($a,b\in\mathbb{Z}$):
$$
\mathcal{L}_{a,b}=\{ae_{4,0}x+byi\,:\,x,y\in\mathbb{Z}\}
$$
Moreover, we know that for $a,b$ in $\mathbb{Z}$, there exists
two integers $s,t$ such that:
$$
a=s\gcd(a,b)
$$ 
and:
$$
b=t\gcd(a,b)
$$
Therefore, any element of $\mathcal{L}_{a,b}$ can be expressed as
($x,y\in\mathbb{Z}$):
$$
ae_{4,0}x+byi=\gcd(a,b)(ae_{4,0}s+bti)
$$
It results that $\mathcal{L}_{a,b}$ is generated by $\gcd(a,b)$ and is
so principal. We conclude that $\mathbb{Z}[i]$ is a principal entire
subring of $\mathbb{C}$.
\\[0.1in]
The expression for the magnitude function is given by:
$$
N(ae_{4,0}+bi)=\sqrt{(ae_{4,0}+bi)(ae_{4,0}-bi)}=\sqrt{a^2-abi+bai-b^2i^2}
$$
$$
N(ae_{4,0}+bi)=\sqrt{a^2+b^2}
$$
Thus, we have:
$$
\ker N=\{0\}
$$
It follows that the expression for the unit function is given by ($a,b\neq 0$):
$$
u(ae_{4,0}+bi)=\frac{1}{\sqrt{a^2+b^2}}(a+bi)
$$
Let consider a natural number $n\geq 3$ which is odd. Let prove by induction
that any family $\{e_{n,0},e_{n,1},\ldots,e_{n,k}\}$ for $k\in \llbracket
0,n-2\rrbracket$ is free. It will prove that the family
$\{e_{n,0},e_{n,1},\ldots,e_{n,n-2}\}$ forms a 
basis of $\mathbb{Z}[S_n]$ when $n$ is an odd positive integer. The case
where $n$ is even can be done in a similar way.
\\[0.1in]
When $k=0$, we have $e_{n,0}=1$ and it is obvious that $\{e_{n,0}\}$ is free.
Let assume that for an integer $k$ such that $k\in\llbracket
0,n-3\rrbracket$, the family $\{e_{n,0},e_{n,1},\ldots,e_{n,k}\}$
is free. If $a_0,a_1,\ldots,a_k,a_{k+1}$ are integers such that:
$$
a_0e_{n,0}+a_1e_{n,1}+\ldots+a_ke_{n,k}+a_{k+1}e_{n,k+1}=0
$$
then since $e_{n,i}=e^i_{n,1}$ for $i\in\mathbb{Z}$, it comes that:
$$
a_0e_{n,0}+a_1e_{n,1}+\ldots+a_ke^k_{n,1}+a_{k+1}e^{k+1}_{n,1}=0
$$
Multiplying the left and the right hand sides of this equation 
by $e^{n-k-1}_{n,1}$, it gives ($e^n_{n,1}=1$):
$$
e^{n-k-1}_{n,1}\left\{a_0e_{n,0}+a_1e_{n,1}+\ldots+a_ke^k_{n,1}\right\}+a_{k+1}=0
$$
$$
e^{n-k-1}_{n,1}\left\{a_0e_{n,0}+a_1e_{n,1}+\ldots+a_ke^k_{n,1}\right\}=-a_{k+1}
$$
By conjugation, we have also:
$$
(e^{n-k-1}_{n,1})^{\star}
\left\{a_0e_{n,0}+a_1e^{\star}_{n,1}+\ldots+a_k(e^k_{n,1})^{\star}\right\}=-a_{k+1}
$$
It results that:
$$
a_0\mathrm{Im}(e^{n-k-1}_{n,1})+a_1\mathrm{Im}(e^{n-k}_{n,1})+\ldots
+a_k\mathrm{Im}(e^{n-1}_{n,1})=0
$$
$$
a_0\mathrm{Im}(e_{n,n-k-1})+a_1\mathrm{Im}(e_{n,n-k})+\ldots
+a_k\mathrm{Im}(e_{n,n-1})=0
$$
Since the family $\{e_{n,0},e_{n,1},\ldots,e_{n,k}\}$
is free from the assumption, from the theorem (\ref{t4.58}), 
the family ($k\in\llbracket 0,n-3\rrbracket$ and $n\geq 3$):
$$
\{\mathrm{Im}(e_{n,n-k-1}),\mathrm{Im}(e_{n,n-k}),\ldots,
\mathrm{Im}(e_{n,n-1})\}
$$ 
is also free. It implies that:
$$
a_0=a_1=\ldots=a_k=0
$$
and so:
$$
a_{k+1}=0
$$
We deduce that if $\{e_{n,0},e_{n,1},\ldots,e_{n,k}\}$
is free for $k\in\llbracket 0,n-3\rrbracket$, then
$\{e_{n,0},e_{n,1},\ldots,e_{n,k},e_{n,k+1}\}$ is also free. It 
achieved the proof by induction of the property that any family
$\{e_{n,0},e_{n,1},\ldots,e_{n,k}\}$ for $k\in \llbracket 
0,n-2\rrbracket$ is free. Therefore, the family
$\{e_{n,0},e_{n,1},\ldots,e_{n,n-2}\}$ forms a 
basis of $\mathbb{Z}[S_n]$ when $n$ is an odd positive integer. A
similar reasoning implies that the family
$\{e_{n,0},e_{n,1},\ldots,e_{n,n-2}\}\setminus\{e_{n,n/2}\}$ forms a 
basis of $\mathbb{Z}[S_n]$ when $n$ is a non-zero even positive integer.
\\[0.1in]
Thus, $\mathbb{Z}[S_n]$ with $n\geq 3$ is a free module such that ($n\geq 3$):
$$
\mathbb{Z}[S_n]=\left\{\begin{array}{ccc}
\mathbb{Z}[e_{n,1},\ldots,e_{n,n/2-1},e_{n,n/2+1},\ldots,e_{n,n-2}]
& \mathrm{if} & n\equiv 0\pmod 2
\\
\mathbb{Z}[e_{n,1},\ldots,e_{n,n-2}]
& \mathrm{if} & n\equiv 1\pmod 2
\end{array}
\right.
$$
If $n\geq 3$ is odd, then the expression for the magnitude function is given by:
$$
N(a_0e_{n,0}+a_1e_{n,1}+\ldots+a_{n-2}e^{n-2}_{n,1})
=\sqrt{{\displaystyle\sum^{n-2}_{i=0}}a^2_ie_{n,0}
+{\displaystyle\sum^{n-2}_{i=0}\sum^{n-2}_{j\neq i}}a_ia_j
(e^i_{n,1}(e^j_{n,1})^{\star}
+(e^i_{n,1})^{\star}e^j_{n,1})} 
$$
If $n\geq 3$ is even, then the expression for the magnitude function
is given by: 
$$
N(a_0e_{n,0}+a_1e_{n,1}+\ldots+a_{n/2-1}e^{n/2-1}_{n,1}+a_{n/2+1}e^{n/2+1}_{n,1}
+\ldots+a_{n-2}e^{n-2}_{n,1})
$$
$$
=\sqrt{{\displaystyle\sum^{n-2}_{i\neq n/2}}a^2_ie_{n,0}
+{\displaystyle\sum^{n-2}_{i\neq n/2}\sum^{n-2}_{j\neq i,n/2}}
a_ia_j(e^i_{n,1}(e^j_{n,1})^{\star}
+(e^i_{n,1})^{\star}e^j_{n,1})} 
$$

\noindent
More generally, let $n\in\mathbb{N}^{\star}$ and let 
$\{e_1,\ldots,e_n\}$ a finite family of elements of
$\mathbb{C}$ such that ($i,j=1,\ldots,n$):
$$
e_n=e_0=1
$$
$$
e_ie_0=e_0e_i=e_i
$$
$$
e_ie_j=e_je_i=e_k
$$
with:
$$
k\equiv\left\{\begin{array}{ccc}
i+j\pmod {\frac{n}{2}} & \mathrm{if} & n\equiv 0\pmod 2
\\
i+j\pmod n & \mathrm{if} & n\equiv 1\pmod 2
\end{array}
\right.
$$
and $e_1,\ldots,e_n$ which are linearly independent over $\mathbb{C}$,
namely:
$$
c_1e_1+\ldots+c_ne_n=0\,\Rightarrow\,c_1=\ldots=c_n=0
$$
Let $A$ be an entire subring of
$\mathbb{C}$ such that:
$$
a^{\star}=a,\,\,\,\,\forall\,\,a\in A
$$
and let $G$ be the subset
$\{e_0,e_1,\ldots,e_{n-1}\}$ with $n\in\mathbb{N}^{\star}$, of
$\mathbb{C}$. Notice that from the properties satisfied
by the $e_i$ with $i=0,1,\ldots,n-1$, $G$ is an abelian group. 
We define $A[S]$ where $S=\mathcal{G}\setminus\{e_0\}$ and
$\mathcal{G}=G$ is the maximal family of linearly independent elements of
$G$ over $\mathbb{C}$, as ($n\in\mathbb{N}^{\star}$):  
$$
A[S]=\{a_0e_0+a_1e_1+\ldots+a_{n-1}e_{n-1}\,:\,a_0,a_1,\ldots,a_n\in A\}
$$
Then, $A[S]$ is an entire subring of $\mathbb{C}$, 
which is generated by $S$ over $A$. We can
notice that the family $\{e_0,\ldots,e_{n-1}\}$ with
$n\in\mathbb{N}^{\star}$ forms a basis 
of $A[S]$. 
\\[0.1in]
In $A[S]$, any element $e_i$ with $i=0,1,\ldots,n-1$ is
invertible since $G$ is a group. The inverse of $e_i$ with
$i=0,1,\ldots,n-1$ is given by $e_{n-i}$. Indeed, we have ($i=0,1,\ldots,n-1$):
$$
e_ie_{n-i}=e_0
$$
Moreover, we have:
$$
e^n_i=e_i\ldots e_i=e_{ni}=e^i_n=e^i_0=e_0
$$
So, $G$ is a cyclic finite subgroup of $A[S]$ of order $n$. 
\\[0.1in]
Notice that ($i=0,1,\ldots,n-1$):
$$
e_i=e^i_1
$$
So, a generator of $\{e_0,\ldots,e_{n-1}\}$ is $e_1$.
Besides, we have ($i=0,1,\ldots,n-1$):
$$
N(e^n_i)=N(e_i)^n=N(e_0)=N(1)=1
$$
Since $||N(e_i)||_{\mathbb{C}}=N(e_i)$, it results that ($i=0,1,\ldots,n-1$):
$$
N(e_i)=1
$$
Therefore, we have ($i=0,1,\ldots,n-1$):
$$
u(e_i)=e_i
$$

\section{{An algebra of entire ring generated by the generators 
of a Lie algebra}}\label{s11}

A more general framework is to consider 
a finite maximal free family $\mathfrak{e}=\{e_1,\ldots,e_n\}$ with
$n\in\mathbb{N}^{\star}$ of generators of a Lie algebra\index{Lie algebra} 
$\mathfrak{g}$ associated\cite{bh},\cite{awk} to a finite-dimensional complex Lie
group\index{Lie group}
$G$ such that ($i,j=1,\ldots,n$ and $n\in\mathbb{N}^{\star}$):
$$
e_n=e_0
$$
$$
e_ie_j=c^k_{ij}e_k
$$
where the Einstein summation over repeated indices $k$ from $k=0$ to
$n-1$ with $n\in\mathbb{N}^{\star}$, is understood. The set $\mathfrak{e}$ is a
basis of the associated Lie algebra $\mathfrak{g}$ of the Lie group
$G$. 
\\[0.1in]
We assume that the elements $c^k_{ij}$s called the structure constants
of $\mathfrak{g}$ with respect to basis $\mathfrak{e}$, belong to an
entire subring $A$ of $\mathbb{C}$ such that $A=\mathrm{Re}(A)$. Denoting 
$\mathfrak{s}=\mathfrak{e}\setminus\{e_n\}=\{e_1,\ldots,e_{n-1}\}$, the set
$A[\mathfrak{s}]$ is the extension of ring $A$ which includes all
linear combinations of elements of $\mathfrak{g}$ with coefficients in
$A$. It is understood that 
$A[\mathfrak{s}]$ is generated by $\mathfrak{s}$. Similarly as in
the case where the Lie algebra $\mathfrak{g}$ 
associated to the Lie group $G$ is replaced by an
abelian group (see above), a star operation is defined on
$A[\mathfrak{s}]$. Since the structure constants $c^k_{ij}$s
with $i,j,k=0,1,\ldots,n-1$, belong to
$A=\mathrm{Re}(A)$, we have ($i,j,k=0,1,\ldots,n-1$ with $n\in\mathbb{N}^{\star}$):
$$
(c^k_{ij})^{\star}=c^k_{ij}
$$
We define another finite family $\{e'_1,\ldots,e'_n\}$ with
$n\in\mathbb{N}^{\star}$ by ($i,j=1,\ldots,n$ and
$n\in\mathbb{N}^{\star}$):
$$
e'_n=e'_0=e_0
$$
$$
e'_i=e_i-v_ie_0=(\delta^j_i-v_i\delta^j_0)e_j
$$
$$
(e'_i)^{\star}=-e'_i
$$
with $v_1,\ldots,v_{n-1}\in\mathbb{C}$ and $v_n=0$ with $n\in\mathbb{N}^{\star}$.
\\[0.1in]
The family $\{e'_1,\ldots,e'_n\}$ is free. Indeed, if
($a'_0,a'_1,\ldots,a'_{n-1}\in\mathbb{C}$ and $n\in\mathbb{N}^{\star}$):
$$
a'_0e'_0+a'_1e'_1+\ldots+a'_{n-1}e'_{n-1}=0
$$
then ($a'_0,a'_1,\ldots,a'_{n-1}\in\mathbb{C}$ and $n\in\mathbb{N}^{\star}$):
$$
a'_0e_0+a'_1(e_1-v_1e_0)+\ldots+a'_{n-1}(e_{n-1}-v_{n-1}e_0)=0
$$
$$
a_0e_0+a_1e_1+\ldots+a_{n-1}e_{n-1}=0
$$
with ($a'_0,a'_1,\ldots,a'_{n-1}\in\mathbb{C}$ and $n\in\mathbb{N}^{\star}$):
$$
a_0=a'_0-(v_1a'_1+\ldots+v_{n-1}a'_{n-1})
$$
and ($i=1,\ldots,n$ and $n\in\mathbb{N}^{\star}$):
$$
a'_i=a_i
$$
Since the family $\{e_1,\ldots,e_n\}$ is free, it implies that
($n\in\mathbb{N}^{\star}$):
$$
a_0=a_1=\ldots=a_{n-1}=0
$$
It gives ($i=1,\ldots,n$ and $n\in\mathbb{N}^{\star}$):
$$
a_i=0
$$
and:
$$
a'_0=a_0=0
$$
Therefore the family $\{e'_1,\ldots,e'_n\}$ is free.
\\[0.1in]
Moreover, it comes that ($i,j=1,\ldots,n$ and $n\in\mathbb{N}^{\star}$):
$$
e_i=e'_i+v_ie_0=(\delta^j_i+v_i\delta^j_0)e'_j
$$
Then, we have ($i,j=1,\ldots,n$ and $n\in\mathbb{N}^{\star}$):
$$
e_ie_j=(\delta^l_i+v_i\delta^l_0)e'_l(\delta^m_j+v_j\delta^m_0)e'_m
$$
$$
e_ie_j=\delta^l_i\delta^m_je'_le'_m+v_j\delta^l_i\delta^m_0e'_le'_m
+v_i\delta^l_0\delta^m_je'_le'_m
+v_iv_j\delta^l_0\delta^m_0e'_le'_m
$$
$$
e_ie_j=e'_ie'_j+v_je'_i+v_ie'_j+v_iv_je^2_0=e'_ie'_j
+(v_i\delta^k_j+v_j\delta^k_i)e'_k
$$
Since $e_ie_j=c^k_{ij}e_k$ for $i,j=1,\ldots,n$ and
$n\in\mathbb{N}^{\star}$, we have also ($i,j=1,\ldots,n$ and
$n\in\mathbb{N}^{\star}$):
$$
e_ie_j=c^k_{ij}(\delta^l_k+v_k\delta^l_0)e'_l
$$
$$
e_ie_j=c^k_{ij}\delta^l_ke'_l+c^k_{ij}v_k\delta^l_0e'_l
$$
$$
e_ie_j=c^k_{ij}e'_k+c^l_{ij}v_l\delta^k_0e'_k
$$
So ($i,j=1,\ldots,n$ and $n\in\mathbb{N}^{\star}$):
$$
e'_ie'_j=d^k_{ij}e'_k
$$
with ($i,j,k=1,\ldots,n$ and $n\in\mathbb{N}^{\star}$):
$$
d^k_{ij}=c^k_{ij}+c^l_{ij}v_l\delta^k_0-(v_i\delta^k_j+v_j\delta^k_i)
$$
Notice that ($i,j=0,1,\ldots,n-1$):
$$
(e'_ie'_j)^{\star}=(d^k_{ij}e'_k)^{\star}=-(d^k_{ij})^{\star}e'_k
$$
$$
(e'_ie'_j)^{\star}=(e'_j)^{\star}(e'_i)^{\star}=e'_je'_i=d^k_{ji}e'_k
$$
It results that ($i,j,k=0,1,\ldots,n-1$):
$$
(d^k_{ij})^{\star}=-d^k_{ji}
$$
The conjugate of $e_i$ for $i=0,1,\ldots,n-1$ with
$n\in\mathbb{N}^{\star}$ is given by ($i=0,1,\ldots,n-1$ and
$n\in\mathbb{N}^{\star}$):
$$
e^{\star}_i=v^{\star}_ie_0-e'_i
$$
$$
e^{\star}_i=v^{\star}_ie_0-(e_i-v_ie_0)
$$
$$
e^{\star}_i=-e_i+(v_i+v^{\star}_i)e_0
$$
It follows that ($i=0,1,\ldots,n-1$ with
$n\in\mathbb{N}^{\star}$):
$$
e_ie^{\star}_i=-e^2_i+(v_i+v^{\star}_i)e_ie_0=-e^2_i
+(v_i+v^{\star}_i)c^k_{i0}e_k
$$
and ($i=0,1,\ldots,n-1$ with
$n\in\mathbb{N}^{\star}$):
$$
e^{\star}_ie_i=-e^2_i+(v_i+v^{\star}_i)e_0e_i=-e^2_i
+(v_i+v^{\star}_i)c^k_{0i}e_k
$$
So, it comes that:
$$
e_ie^{\star}_i=e^{\star}_ie_i+(v_i+v^{\star}_i)(c^k_{i0}-c^k_{0i})e_k
$$ 
It results that the function $N$ is not well defined on
$A[\mathfrak{s}]$ unless ($i=0,1,\ldots,n-1$; $k=0,\ldots,n-1$ with
$n\in\mathbb{N}^{\star}$):
$$
c^k_{i0}=c^k_{0i}
$$
Or, since $e_ie_0=-e_0e_i$, we have $c^k_{i0}=-c^k_{0i}$. Therefore, the function $N$ is well defined on $A[\mathfrak{s}]$ if ($i=0,1,\ldots,n-1$;
$k=0,1,\ldots,n-1$ with $n\in\mathbb{N}^{\star}$):
$$
c^k_{i0}=c^k_{0i}=0
$$
But, in this case, since $\mathfrak{g}$ is Lie algebra, we have:
$$
e_ie^{\star}_i=e^{\star}_ie_i=-e^{\star}_ie_i
$$
meaning that:
$$
e_ie^{\star}_i=e^{\star}_ie_i=0
$$
It involves that the function $N$ which can be defined on any element of $A[\mathfrak{s}]$ which commutes with its conjugate cancels. We conclude that $N$ is either degenerate on a subset of  $A[\mathfrak{s}]$ or is not defined on $A[\mathfrak{s}]$.

\section{{Acknowledgements}}

The author would like to thank Aleks Kleyn for helpful comments. 

\printindex

\end{document}